\documentclass[11pt,a4paper,reqno]{article}  
\usepackage{ifthen,amssymb,mathrsfs,amsthm,a4wide,color,mathtools}
\usepackage{enumitem}
\setlist[itemize]{noitemsep,topsep=0pt, itemsep=-2pt}
\setlist[enumerate]{noitemsep,topsep=0pt, itemsep=-2pt}
\mathtoolsset{showonlyrefs}
\let\spacecal=\mathscr
\usepackage[all]{xypic}
\usepackage[backref=page]{hyperref}
\DeclareMathAlphabet{\mathpzc}{OT1}{pzc}{m}{it}
\hypersetup{
 colorlinks,
 citecolor=green,
 linkcolor=blue,
 urlcolor=Blue}
\makeatletter
\@namedef{subjclassname@2020}{%
  \textup{2020} Mathematics Subject Classification}
\makeatother
\definecolor{maroon}{rgb}{0.5, 0.0, 0.0}
\definecolor{verde}{rgb}{0, 0.5, 0.5}

\newcommand{\red}[1]{\textcolor{red}{#1}}

\newcommand{\marginnote}[1]{\ifthenelse{\isodd{\thepage}}{\normalmarginpar}
{\reversemarginpar}\marginpar{\fbox{\parbox{15mm}{\sloppy\footnotesize \red{#1}}}}}
\let\tempone\itemize
\let\temptwo\enditemize
\let\temponenum\enumerate
\let\temptwonum\endenumerate
\renewenvironment{itemize}{\tempone\addtolength{\itemsep}{0.2\baselineskip}}{\temptwo}
\renewenvironment{enumerate}{\temponenum\addtolength{\itemsep}{0.2\baselineskip}}{\temptwonum}
\let\tempsec\section
\renewcommand{\section}{\par\medskip\tempsec}

\newcommand{\qee}{\parbox{5pt}{\hfill}\hfill $\triangle$}

\newenvironment{rem}{\begin{remqee}}{\qee\end{remqee}}
\newtheorem{thm}{Theorem}[section] 
\newtheorem{conj}[thm]{Conjecture}

\newtheorem{lemma}[thm]{Lemma} 
\newtheorem{prop}[thm]{Proposition}
\newtheorem{defin}[thm]{Definition}
\theoremstyle{definition}

\theoremstyle{remark}
\newtheorem{remqee}[thm]{Remark}

\newtheorem{example}[thm]{Example}
\newenvironment{remark}{\begin{remqee}}{\qee\end{remqee}}
\numberwithin{equation}{section}

\newcommand\rk{\operatorname{rk}}

\newcommand\Spec{\operatorname{Spec}}

\newcommand\iso{\kern.35em{\raise3pt\hbox{$\sim$}\kern-1.1em\to}\kern.3em}

\newcommand\Sym{{\mathcal{S}ym}}
\newcommand\F{{\mathcal F}}
\newcommand\Oc{{\mathcal O}}
\newcommand\Pc{{\mathcal P}}
\newcommand\M{{\mathcal M}}

\newcommand\Z{{\mathbb Z}}

\newcommand\C{{\mathbb C}}
\newcommand\Q{{\mathbb Q}}
\newcommand\K{{\mathbb K}}
\newcommand\td{\operatorname{td}}
\newcommand\ch{\operatorname{ch}}
\newcommand\U{{\mathcal U}}

\newcommand\Nc{{\mathcal N}}

\newcommand\Lcl{{\mathcal L}}
\newcommand\Ec{{\mathcal E}}
\newcommand\Fc{{\mathcal F}}
\newcommand\Ac{{\mathcal A}}

\newcommand\Ps{{\mathbb P}}

\newcommand\Gc{{\mathcal G}}
\newcommand\Xcal{{\spacecal X}}
\newcommand\Ycal{{\spacecal Y}}
\newcommand\Sc{{\spacecal S}}
\newcommand\Tc{{\spacecal T}}
\newcommand\Zc{{\spacecal Z}}
\newcommand\Vc{{\spacecal V}}
\newcommand\Wc{{\spacecal W}}
\newcommand\Hcal{{\spacecal H}}

\newcommand\Hc{{\mathcal H}}
\newcommand\Ucal{{\spacecal U}}

\newcommand\Qc{{\mathcal Q}}

\newcommand\Jc{{\mathcal J}}

\newcommand\Homsh{{{\mathcal H}om}}

\newcommand\Xf{{\mathfrak X}}

\newcommand\Sf{{\mathfrak S}}

\newcommand\nf{{\mathfrak n}}
\newcommand\Mf{{\mathfrak M}}

\newcommand\SSpec{\mathbb{S}\mathrm{pec}\,}

\newcommand\SHilb{\mathbb{SH}\mathrm{ilb}\,}

\newcommand{\gr}{\operatorname{gr}\,}

\newcommand{\tgr}{\widetilde{\operatorname{gr}}\,}

\newcommand{\cl}{\operatorname{cl}}
\newcommand{\redu}{{ }_{\mbox{\tiny bos}}}
\renewcommand{\Im}{\operatorname{Im}}

\newcommand{\Uf}{\mathfrak U}
\newcommand{\Ff}{\mathfrak F}

  \makeindex

\begin{document}
\begin{center} \bf \Large  MODULI OF STABLE SUPERMAPS \end{center}

\thispagestyle{empty}  
\begin{center}{\sc Ugo Bruzzo}$^{ab}$    and {\sc Daniel Hern\'andez Ruip\'erez}$^{cd}$ \\[5pt]
\small
$^a$ 
Departamento de Matem\'atica, Instituto de Ci\^encias Exatas,  \\ Universidade Federal de Minas Gerais, Av.~Ant\^onio Carlos 6627,  \\ Belo Horizonte,
MG 30123‑970, Brazil \\
$^b$ IGAP (Institute for Geometry and Physics), Trieste, Italy  \\
$^c$ Departamento de Matem\'aticas  and IUFFYM (Instituto de F\'\i sica \\ Fundamental y   Matem\'aticas), Universidad de Salamanca,  \\ Plaza de la Merced 1-4, 37008 Salamanca, Spain \\
$^d$ Real Academia de Ciencias Exactas, F\'isicas y Naturales, Spain\\[5pt] 
Email: {\tt ubruzzo@ufmg.br,  ruiperez@usal.es}
\end{center}

\bigskip\bigskip
\begin{abstract} 
We review the notion of stable supermap from   SUSY curves to a fixed target superscheme $\Ycal$, and prove that  when $\Ycal$ is projective,
 stable supermaps are parameterized by a Deligne-Mumford superstack $ \Sf\Mf(\Ycal,\beta)$ with superschematic and separated diagonal. We characterize the bosonic reduction of this moduli superstack and see that it has a surjective morphism onto the moduli stack of stable maps from spin curves to the bosonic reduction $Y$ of $\Ycal$, whose fibers are linear schemes; for this reason, $ \Sf\Mf(\Ycal,\beta)$ is not proper unless such linear schemes reduce to a point. Using Manin-Penkov-Voronov's super Grothendieck-Riemann-Roch theorem we also make a formal computation of the virtual dimension of $ \Sf\Mf(\Ycal,\beta)$, which agrees with the characterization of the bosonic reduction just mentioned and with the dimension formula for the case of bosonic target existing in the literature.

\end{abstract}

\vfill
\noindent\parbox{\textwidth}{\small \baselineskip=14pt
\noindent\parbox{.8\textwidth}{\hrulefill} \\
\noindent February 19, 2026 \\
\noindent AMS Subject classification  (2020)  {Primary: 14D23; Secondary: 14A20, 14H10, 14M30, 81T30, 83E30}  
\noindent U.B.: Research partly supported by PRIN 2022BTA242 ``Geometry of algebraic structures: moduli, invariants, deformations,''  INdAM-GNSAGA, and CNPq Bolsa de Produtividade em Pesquisa B 305343/2025-4.  \\
 D.H.R: Research partly supported by the  research projects  ``Espacios finitos y functores integrales'', MTM2017-86042-P (Ministerio de Econom\'{\i}a, Industria y Competitividad) and GIR 4139 of the University of Salamanca.
 }

 \newpage

\tableofcontents 
  
 \bigskip

\section{Introduction} 

Stable supermaps are a natural supergeometric generalization of the notion of stable map. A definition of stable supermap was given in \cite{BrHRMa23} in the algebro-geometric formalism, as a morphism from a SUSY (supersymmetric) curve with punctures to a fixed target superscheme, satisfying a stability condition (another definition was given by Ke\ss ler, Sheshmani and Yau in \cite{KSY-2022} in the symplectic approach, considering as target an ordinary (non-super) symplectic manifold). According to \cite{BrHRMa23}, there exists a moduli space of stable supermaps 
(for which, unless the contrary is stated, we understand that the target is a superscheme $\Ycal$), which is a superstack.
The data to define stable supermaps include a fixed supercycle $\beta$ in $\Ycal$; we momentarily denote by $ \Sf\Mf(\Ycal,\beta)$ the resulting superstack, omitting   genus and punctures from the notation.

 In this paper we study the superstack $ \Sf\Mf(\Ycal,\beta)$. The main results are three:
\begin{itemize} 
\item In section \ref{alg} we prove that it is a  Deligne-Mumford  superstack \cite[Def.\! 3.7]{BHRSS25}  with superschematic and separated diagonal  (Felder, Kazhdan and Polishchuk in \cite{FKPpubl} constructed the Deligne-Mumford superstack of stable SUSY curves. Our result implies that that superstack has as a superschematic and separate diagonal).
\item We identify its bosonic reduction as a fibration over the stack of stable maps from spin curves to $\Ycal$ (Section \ref{red}). A consequence of this fact is that the superstack $ \Sf\Mf(\Ycal,\beta)$  may not  be proper (but it is proper in some cases).
\item We conjecture a formula for the the virtual dimension of the superstack $ \Sf\Mf(\Ycal,\beta)$  by using    Manin-Penkov-Voronov's super Grothendieck-Riemann-Roch theorem (Section \ref{dim}). This computation needs to be supported by a perfect obstruction theory for the superstacks $ \Sf\Mf(\Ycal,\beta)$, and by  the resulting deformation theory. This will form the object of a future paper \cite{POT}. The virtual dimension we compute coincides with the formula   in
\cite{KSY-2022} when the target is bosonic (an ordinary scheme $Y$).
\end{itemize}

Among the necessary prerequisites, in Section \ref{alg}, following \cite{BrHRMa23},  we recall the basic definitions about supercycles in superschemes, and state the existence of proper push-forwards of the groups of supercycles modulo rational equivalence.  In Section \ref{char} we recall from the work of Manin-Penkov-Voronov
\cite{VMP}  the construction of a K-theory for superschemes and  the statement of  the super
Grothendieck-Riemann-Roch theorem, and give a detailed proof of it. 
The notions of algebraic supergeometry that we shall need were resumed in the introductory parts of \cite{BrHR21} and \cite{BrHRPo20}.   An introduction to superstacks is given in the companion paper \cite{BHRSS25}; see also \cite{CodViv17}. 


\smallskip
\noindent{\bf Acknowledgments.} We thank  Barbara Fantechi, Margarida Melo, Emanuele Pavia  and Andrea Ricolfi for useful discussions. D.H.R.~thanks GNSAGA-INdAM for support, and SISSA for hospitality and support. 

\section{The superstack of stable supermaps} \label{alg}

Stable supermaps, that is, maps from SUSY curves to a fixed target superscheme satisfying a stability condition, were defined in \cite{BrHRMa23}.
This defines a category fibered in groupoids which in  \cite{BrHRMa23} was shown to be a superstack.  In this section, after recalling the relevant definitions, we first  show that the superstack of stable supermaps is algebraic  and then Deligne-Mumford, and has a separated and schematic diagonal.  At first we shall review definitions and results about prestable and stable SUSY curves. We shall also recall the definition of supercycle and state the existence of proper push-forwards for the groups of supercycles modulo rational equivalence. 

\subsection{Stable and prestable SUSY curves}

Prestable and stable SUSY curves and morphisms between them were introduced in  \cite{FKPpubl} and further studied in \cite{BrHRMa23}.
We first define supercurves with punctures and then SUSY curves with punctures;  SUSY curves are supercurves with an additional structure (the superconformal structure). This differs from  the terminology of \cite{FKPpubl}, where
``supercurve''  refers to what we call a  SUSY curve. 

\begin{defin}\label{def:prestablecurve} A supercurve of arithmetic genus $g$ is a proper and Cohen-Macaulay morphism $f\colon \Xcal \to \Sc$ of superschemes of relative dimension $1|1$ (see \cite[Def.\! 2.21]{BrHRMa23}) whose bosonic fibers have arithmetic genus $g$.

 A prestable supercurve of arithmetic genus $g$ with punctures over a superscheme $\Sc$ is a supercurve $f\colon \Xcal \to \Sc$, together with:
\begin{enumerate}
\item a collection of disjoint closed sub-superschemes $\Xcal_i\hookrightarrow \U$ ($i=1,\dots,\nf_{NS}$), where $\U$ is the smooth locus of $f$, such that $\pi\colon \Xcal_i\to\Sc$ is an isomorphism for every $i$. These superschemes are called   \emph{Neveu-Schwartz (NS) punctures};
\item a collection of disjoint Cartier divisors $\Zc_j$ of relative degree 1 ($j=1,\dots,\nf_{RR}$), contained in $\U$. They are called   \emph{Ramond-Ramond (RR) punctures}. 
\end{enumerate}
These data must satisfy the following condition: if for every bosonic fiber $X_s$ of $f\colon \Xcal \to \Sc$ we write $x_{s,i}=X_i\cap X_s$ and $z_{s,j}=Z_j\cap X_s$,  then the pair $(X_s, D_s)$, where   $$D_s= \{x_{s,1}\dots x_{s,\nf_{NS}},z_{s,1}\dots, z_{s,\nf_{R}}\},$$ is a prestable $(\nf_{NS}+\nf_{RR})$-pointed curve of 
arithmetic genus $g$.\footnote{Here we adopt the definition of prestable curve as a connected  curve with at most nodal singularities, see \cite[13.2.2]{Ol16}.}

A prestable SUSY curve of arithmetic genus $g$ with punctures over a superscheme $\Sc$ is a  prestable supercurve $f\colon\Xcal\to\Sc$ of arithmetic genus $g$ with punctures together with
an epimorphism $\bar\delta\colon \Omega_f\to \omega_f(\Zc)$ of $\Oc_\Xcal$-modules, where $\Zc=\sum \Zc_j$, satisfying the conditions of \cite[Def.\! 5.8]{BrHRMa23}, that is, the composition
$$
\ker\bar\delta \hookrightarrow \Omega_f\xrightarrow{d} \Omega_f\wedge \Omega_f\xrightarrow{\bar\delta\wedge\bar\delta} \omega_f^{\otimes 2}(2\Zc)
$$
yields an isomorphism
$\ker\bar\delta \iso \omega_f^{\otimes 2}(\Zc)$
(see \cite[Prop. 3.5]{BrHR21}).

Morphisms of prestable supercurves (resp.~SUSY curves) $(f\colon\Xcal \to \Sc, \{\Xcal_i\}, \{\Zc_j\}) \to (f'\colon\Xcal' \to \Sc, \{\Xcal'_i\}, \{\Zc'_j\})$ over $\Sc$ are defined as morphisms of  $\Sc$-superschemes $g\colon \Xcal \to \Xcal'$  preserving the punctures (resp.~preserving the punctures and the morphism $\bar\delta$).
By this we mean:
\begin{itemize}
\item $g$ maps surjectively $\{\Xcal_i\}$ onto $\{\Xcal'_i\}$;
\item the inverse  image $g^{-1}$ establishes  a bijection between the RR punctures $\{\Zc'_j\}$ of  $\Xcal'$ and the RR punctures of the $\Xcal$ contained in the inverse image $g^{-1}(\U')\hookrightarrow \U$ of the smooth locus $\U'$ of $f'$;
\item the morphism $g^\ast\colon g^\ast\Omega_{f'}\to \Omega_f$ maps $\ker g^\ast(\bar\delta')$ to $\ker \bar\delta$,  that is,  the diagram
$$
\xymatrix{
& g^\ast \ker\bar\delta'\ar[d]\ar[r]& g^\ast\Omega_{f'}\ar[d]^{g^\ast}\ar[r]^(.3){g^\ast(\bar\delta')} & g^\ast\omega_{f'}(g^{-1}(\Zc'))\ar[r]& 0
\\
0\ar[r]&\ker \bar\delta\ar[r]&\Omega_f\ar[r]^{\bar\delta}&\omega_f(\Zc)\ar[r]&0
}
$$
commutes.
\end{itemize}
\end{defin}
The terminology ``Neveu-Schwarz'' and ``Ramond-Ramond'' comes from string theory; here we are not dealing with those aspects.

\begin{rem}
Definition \ref{def:prestablecurve} coincides  with \cite[Def.\,3.9]{BrHR21} for smooth supercurves. We know from \cite[Rem.\! 5.13]{BrHRMa23} that in the smooth case, the number of RR punctures of the two curves  must be the same. In the singular case, $\Xcal$ may have also RR punctures not contained in $g^{-1}(\U')$. As we shall see, if $g\colon \Xcal \to \Xcal'$ is a morphism between prestable curves over a field $k$ and $x\in X'$ is a node, we may have RR punctures in $\Xcal$ contained in the fibre $g^{-1}(x)$, so that $\Xcal$ has more RR punctures than $\Xcal$. Concerning  the NS punctures the definition implies that $\nf_{NS}\geq\nf'_{NS}$. 
\end{rem}

To define stable SUSY curves we make use of two facts proved in \cite{FKPpubl}. 
The first is that given a prestable SUSY curve  $f\colon \Xcal\to\Sc$ (Definition \ref{def:prestablecurve}), every NS puncture $\Xcal_i\hookrightarrow \Xcal$  corresponds to a Cartier positive superdivisor $\Wc_i\hookrightarrow \Xcal$ with the same support \cite[Def.\! 2.13]{FKPpubl}. The second   is that if $i\colon \Ucal\hookrightarrow \Xcal$ is the smooth locus of $f\colon \Xcal\to\Sc$, the sheaf
$$
\omega_{\Xcal/\Sc}^2:= i_\ast i^\ast \omega_{\Xcal/\Sc}^{\otimes 2}=i_\ast (\omega_{\Ucal/\Sc}^{\otimes 2})\,,
$$
where $\omega_{\Xcal/\Sc}=\omega_f$ is the relative dualizing sheaf, is locally free, i.e., it  is a  {line bundle} \cite[Thm.\! 4.1]{FKPpubl}.

\begin{defin}\label{def:stablecurve} A stable SUSY curve with punctures is a prestable SUSY curve $f\colon \Xcal\to\Sc$ with punctures (Definition \ref{def:prestablecurve}), such that the line bundle
$$
\omega_{\Xcal/\Sc}^2(\Wc+\Zc)\,,
$$
where $\Wc=\sum_i\Wc_i$ and $\Zc=\sum_j\Zc_j$, is relatively strongly ample with respect to $f\colon\Xcal\to\Sc$ (\cite[Subsect.\! 2.5]{BrHRPo20} or \cite[Def.\! 4.2]{FKPpubl}).
\end{defin}

Following \cite{FKPpubl} we  can now   define the category fibered in groupoids  of stable  SUSY curves with $\nf_{NS}$ NS punctures and $\nf_{RR}$ RR punctures $p\colon \Mf_{g,\nf_{NS},\nf_{RR}}\to \Sf$. 

\begin{prop}{\rm \cite{FKPpubl} }\label{prop:moduli}  
The category fibered in groupoids  $p\colon \Mf_{g,\nf_{NS},\nf_{RR}}\to \Sf$ 
of stable SUSY curves with $\nf_{NS}$ NS punctures and $\nf_{RR}$ RR punctures  is a  proper and smooth Deligne-Mumford superstack.
\qed
\end{prop}

\subsection{Supercycles}
As in the case of stable maps, the definition of stable supermap involves fixing a 1-cycle class in the target. So we need a theory of cycles for superschemes. Some basic elements for such a theory were laid down in \cite{BrHRMa23}; here we recall the basic definitions and the existence of proper pushforwards.

We start by giving $\Z^2$ a superring structure   writing it as $\Z\oplus\Pi\Z$ (here $\Pi$ is the parity change operator), so that\footnote{Diffferently from the convention in \cite{BrHRMa23}, we write  $m-\Pi n$ instead of $m+\Pi n$, and analogously for supercycles in the next definition, to conform with the convention in \cite{VMP}.}
$$(m-\Pi n)(m'-\Pi n') = (mm'+nn' - \Pi (mn'+m'n)).$$
 
 \begin{defin}  Ler $\Xcal$ be a  superscheme.
An $h$-supercycle of $\Xcal$ is a finite sum
 $$
\alpha={\textstyle \sum_i} (m_i -  {\Pi }n_i) [Y_i]
$$
 where $m_i-\Pi n_i \in\Z^2$ and the $Y_i$ are closed subvarieties of $X$ of dimension $h$. The set $Z_h(\Xcal)$  of   $h$-supercycles is a free $\Z_2$-graded module over $\Z^2$.
The group of supercycles of $\Xcal$ is the bigraded $\Z^2$-module
$$
Z_\bullet (\Xcal)=\bigoplus_{h=0}^m Z_h(\Xcal) = Z_\bullet (X) \oplus \Pi Z_\bullet (X)\,.
$$\end{defin}
$Z_\bullet (\Xcal)$   has a natural ring structure induced by  the ordinary intersection product.
This definition of the group of supercycles should be compared with 
   Manin-Penkov-Voronov's definition of K-theory rings for superschemes, $K^S(\Xcal) = K(X) \oplus \Pi K(X)$ (see Section \ref{char} and \cite{VMP}).

\begin{defin}\ 
\begin{enumerate}
\item A supercycle
  $\alpha\in Z_h(\Xcal)$ is rationally equivalent to zero if there are $t$  sub-supervarieties $\delta_i\colon\Wc_i\hookrightarrow\Xcal$  of even dimension $h+1$ and   odd dimension  $s=0$ or $1$   and nonzero rational even superfunctions $g_i\in \K(\Wc_i)^\ast$ such that
$$
\alpha=\sum_{i=0}^t \delta_{i\ast}\operatorname{div}(g_i)\,.
$$
\item  The $\Z^2$-module of $h$-supercycles modulo rational equivalence  is
 $$
A_h(\Xcal)= Z_h(\Xcal)/W_h(\Xcal)
$$
where $W_h(\Xcal)\subset Z_h(\Xcal)$ is the  graded $\Z^2$-submodule  formed by the 
 $h$-supercycles rationally equivalent to zero.
 \end{enumerate}
 \end{defin}
 \begin{prop}
 If $f\colon \Xcal \to \Ycal$ is a proper morphism of superschemes, there is a pushforward morphism $$f_\ast\colon A_\bullet(\Xcal) \to A_\bullet(\Ycal) \,.$$
\end{prop}

\subsection{Stable supermaps}

Actually the  definition of stable supermap  we use in this paper  would be   different from  the one given in  \cite[Def.~6.1]{BrHRMa23} but it coincides with it when the target is a projective superscheme.  So we consider only the case of a projective target $\Ycal$.

\begin{defin}\label{def:stablemap2} Let $\Ycal$ be a projective superscheme and $\beta\in A_1(\Ycal$) a rational equivalence class of supercycles \cite{BrHRMa23}. 
A stable supermap $F$ to $\Ycal$  over a superscheme $\Sc$ of class $\beta$ with $\nf_{NS}$ NS punctures and $\nf_{RR}$ RR punctures is given by the following set of data:
\begin{enumerate}
\item A prestable SUSY curve $(f\colon\Xcal\to\Sc, \{\Xcal_i\},\{\Zc_j\}, \bar\delta)$ over $\Sc$ with $\nf_{NS}$ Neveu-Schwarz punctures and $\nf_{RR}$ Ramond-Ramond punctures (Definition \ref{def:prestablecurve}).
\item A superscheme morphism $\phi\colon\Xcal\to \Ycal$ such that $\phi_\ast[\Xcal_s]=\beta$ for every closed point $s\in S$.
\item If we denote as above by $\Wc_i$ the superdivisor corresponding to the Neveu-Schwarz puncture $\Xcal_i$ according to \cite[Def.\! 2.13]{FKPpubl}, the even line bundle
$$
\Lcl_F:=\omega_{\Xcal/\Sc}^2\big(\sum_{1\le i\le \nf_{NS}} \Wc_i + \sum_{1\le j\le \nf_{RR}} \Zc_j\big)\otimes \phi^\ast \Lcl_\Ycal^{\otimes 3},
$$ 
where $\Lcl_\Ycal = \Oc_\Ycal(1)$,
is  strongly relatively ample with respect to $f\colon\Xcal\to\Sc$ (\cite[Subsect.\! 2.5]{BrHRPo20} or \cite[Def.\! 4.2]{FKPpubl}).
\end{enumerate}
\end{defin}

When $\Ycal$ reduces to a   point, we recover the definition of stable SUSY curve with punctures. 
We now fix nonnegative integers $g$, $\nf_{NS}$ and $\nf_{RR}$ and a 1-supercycle $\beta\in A_1(\Ycal)$.  
As in \cite[Def.\! 6.5]{BrHRMa23} we have:

\begin{defin}\label{def:CFGsmaps}  The category fibered in groupoids 
$$
\Sf\Mf_{g,\nf_{NS},\nf_{RR}}(\Ycal,\beta) \xrightarrow{p} \Sf
$$
of $\beta$-valued stable supermaps of arithmetic genus $g$ with $\nf_{NS}$ Neveu-Schwarz  punctures and $\nf_{RR}$ Ramond-Ramond punctures into $\Ycal$ 
is given by the following data: 
\begin{itemize}\item Objects are $\beta$-valued stable supermaps $\Xf=((f\colon\Xcal\to\Sc, \{\Xcal_i\},\{\Zc_j\}, \bar\delta), \phi)$\footnote{We use a simplified notation here; some of the relevant  data are left implicit.} over a superscheme $\Sc$ into $\Ycal$, of arithmetic genus $g$, with $\nf_{NS}$ Neveu-Schwarz punctures and $\nf_{RR}$ Ramond-Ramond punctures. 
\item Morphisms are cartesian diagrams
\begin{equation}\label{eq:cartesian}
\xymatrix{
\Xf'\ar[r]^\Xi\ar[d]^g & \Xf \ar[d]^f\\
\Tc\ar[r]^\xi & \Sc
}
\end{equation}
that is, diagrams inducing an isomorphism $\Xcal'\iso \xi^\ast \Xcal:=\Xcal\times_{\Sc}\Tc$ of superschemes over $\Tc$ compatible with the punctures and the morphisms $\bar\delta$ and such that $\phi'=\phi\circ\Xi$. 
The functor $p$  maps a stable supermap to the base superscheme  and the morphism $\Xi$ to the base morphism $\xi$. The pullback $\xi^\ast\Xf\to\Tc$ is given by the fiber product $f_\Tc\colon\Xcal\times_{\Sc}\Tc\to\Tc$ and the fiber products of the data $\{\Xcal_i\},\{\Zc_j\}, \bar\delta$ (thus providing  a natural ``cleavage'').\footnote{The fiber product of $\bar\delta$ is a derivation of the same kind due to \cite[Prop.\! 2.26]{BrHRMa23}.} \end{itemize}
\end{defin}

\begin{remark}
Condition (3) of Definition \ref{def:stablemap2} implies that the prestable SUSY curve associated with a stable supermap is projective.
\end{remark}

Since \'etale descent data for stable supermaps  and  for morphims to $\Ycal$ are effective, we   see that descent data for $\Sf\Mf_{g,\nf_{NS},\nf_R}(\Ycal,\beta)$ are effective. One   checks that the isomorphisms between two objects of $\Sf\Mf_{g,\nf_{NS},\nf_R}(\Ycal,\beta)$ form a sheaf in the \'etale topology of superschemes, so that the category fibered in groupoids $\Sf\Mf_{g,\nf_{NS},\nf_R}(\Ycal,\beta)$ is a \emph{superstack} --- see \cite{BHRSS25}.

\subsection{The superstack of stable supermaps is algebraic}

In this subsection we prove that 
the superstack of stable supermaps is algebraic \cite[Def.\! 3.14]{BHRSS25}. To do that we follow the steps of the proof of the analogous statement  about stable maps of schemes given in \cite[Thm.\! 13.3.15]{Ol16}. To simplify notation we   write $\Sf\Mf(\Ycal,\beta):=\Sf\Mf_{g,\nf_{NS},\nf_{RR}}(\Ycal,\beta)$ omitting references to genus and punctures.

The first step is the following:
\begin{lemma}\label{lem:represdiag}
Every morphism $\Sc\to \Sf\Mf(\Ycal,\beta)$ from a superscheme $\Sc$ is superschematic. Equivalently, the diagonal 
$$
\Sf\Mf(\Ycal,\beta) \xrightarrow{\Delta} 
\Sf\Mf(\Ycal,\beta)\times \Sf\Mf(\Ycal,\beta)
$$
is superschematic \cite[Lemma 3.6]{BHRSS25}.
\end{lemma}
\begin{proof}  
Using results from \cite{FKPpubl}
we   adapt to our setting  the proof  of the similar statement given in \cite{Ol16}. Let $F:=(f\colon\Xcal\to\Sc, \{\Xcal_i\},\{\Zc_j\}, \bar\delta, \phi)$ be a stable supermap of genus $g$ over a superscheme $\Sc$. Consider the functor that associates with an $\Sc$-superscheme $\Tc\to\Sc$ the set of the automorphisms $\Xcal_\Tc \iso \Xcal_\Tc$ over $\Tc$ as prestable SUSY curves with punctures.
Using the existence of the superscheme of isomorphisms  \cite[Cor.\! 4.27]{BrHRPo20}, and proceeding as in \cite[5.2]{FKPpubl}, one shows that the functor is representable by an $\Sc$-super\-scheme $\Vc$. There is a canonical automorphism
$$
\sigma\colon \Xcal_\Vc \iso \Xcal_\Vc
$$
of prestable SUSY curves with punctures over $\Vc$. Let us consider the cartesian diagram
$$
\xymatrix{ \Zc \ar@{^(->}[r] \ar[d] & \Xcal_\Vc \ar[d]^{(\phi, \phi\circ\sigma)}
\\
\Ycal \ar@{^(->}[r]^(.4){\Delta} & \Ycal\times \Ycal\,.
}
$$
The fiber product $\spacecal{P}$ of the diagram
$$
\xymatrix{
& \Sc\ar[d]^{(F, F)}\\
\Sf\Mf(\Ycal,\beta) \ar[r]^(.35)\Delta &
\Sf\Mf(\Ycal,\beta)\times \Sf\Mf(\Ycal,\beta)
}
$$
is the subfunctor of $\Vc$ that with a point $t\colon \Tc\to \Vc$ associates the one-point set if the pullback $\Zc_\Vc\to \Xcal_\Vc$ is an isomorphism,  and the empty set otherwise.

Let $\Jc $ be the ideal sheaf of $\Zc\hookrightarrow \Xcal_\Vc$. By \cite[Thm.\! 2.35 and Prop.\! 3.7]{BrHRPo20} there is a positive integer $r$ such that $\Jc\otimes \Lcl_F^{\otimes r}$ and $\Lcl_F^{\otimes r}$ are generated by their global sections and $f_{\Vc \star}(\Lcl_F^{\otimes r})$ is a locally free graded $\Oc_\Vc$-module of finite rank, whose formation commutes with arbitrary base change. Then $\spacecal{P}$ can also be viewed as the subfunctor of $\Vc$ that to a point $t\colon \Tc\to \Vc$ associates the one-point set if the induced map $t^\ast(f_{\Vc\ast}(\Jc\otimes \Lcl_F^{\otimes r})) \to t_\ast f_\Vc^\ast \Lcl_F^{\otimes r}$  is the zero map and the empty set otherwise. We can see now that $\spacecal{P}$ is represented by a closed sub-superscheme of $\Vc$. Let $\Vc=\bigcup_i\Vc_i$ be an open covering such that $f_{\Vc_i\star}(\Lcl_F^{\otimes r})$ is free and choose a trivialization $\Oc_{\Vc_i}^{m|n}\iso [f_{\Vc \star}(\Lcl_F^{\otimes r})]_{|\Vc_i}$. Then, a choice of homogeneous generators of the restrictions of $f_{\Vc\ast}(\Jc\otimes \Lcl_F^{\otimes r})$ to  $\Vc_i$ yields homogeneous elements $g_i\in \Oc_{\Vc_i}^{m|n}$, and $\spacecal{P}_{|\Vc_i}$ is given by their zero locus.
\end{proof}

Let us consider now, for every integer $r$, the substack $\Uf_r$ of $\Sf\Mf(\Ycal,\beta)$ whose objects over a superscheme $\Sc$ are the stable supermaps $F:=(f\colon\Xcal\to\Sc, \{\Xcal_i\},\{\Zc_j\}, \bar\delta, \phi)$ satisfying the following condition:

  {\em  For every geometric point $s\in S$, the restriction of sheaf $\Lcl_F^{\otimes r}$ to the fiber $\Xcal_s$ is very ample and acyclic.}

This condition implies that $f_\ast(\Lcl_F^{\otimes r})$ is locally free of finite rank and that its formation commutes with arbitrary base change \cite[Sec.\! 2,3]{BrHRPo20}. It  also implies that $\Lcl_F^{\otimes r}$ is relatively very ample on $f\colon X \to \Sc$. 

\begin{lemma}\label{lem:ur} $\Uf_r$ is an open substack of $\Sf\Mf(\Ycal,\beta)$, and 
$$
\Sf\Mf(\Ycal,\beta) = \bigcup_{1\leq r} \Uf_r\,.
$$
\end{lemma}
\begin{proof} By \cite[Prop.\! A.2]{FKPpubl} the sheaf $(\Lcl_F^{\otimes r})_{|\Xcal_s}$ is very ample on $\Xcal_s$  if and only if its bosonic reduction is. We can then apply \cite[4.7.1]{EGAIII-II} to conclude that the condition of $\Lcl_F^{\otimes r}$ being very ample on the fibers is open. By cohomology and base change \cite[Sec.\! 2,3]{BrHRPo20}, the vanishing of $H^1$ on fibers is also an open condition. One then has that $\Uf_r$ is an open substack of $\Sf\Mf(\Ycal,\beta)$. The second part is clear since $\Lcl_F$ is relatively ample.
\end{proof}

\begin{thm}\label{thm:algstack}  
The category fibered in groupoids
$\Sf\Mf_{g,\nf_{NS},\nf_{RR}}(\Ycal,\beta) \xrightarrow{p} \Sf$
 is an algebraic superstack with superschematic and separated diagonal.
\end{thm}
\begin{proof} By Lemma \ref{lem:ur}, to prove that $\Sf\Mf(\Ycal,\beta)$ is an algebraic superstack it is enough to show that the stacks $\Uf_r$ are algebraic. Since they are open substacks of $\Sf\Mf(\Ycal,\beta)$, Lemma \ref{lem:represdiag} implies that every morphism from a superscheme to $\Uf_r$ is representable. Then we need only to prove that there is a surjective smooth morphism $\Ucal\to \Uf_r$ from a superscheme $\Ucal$.

Notice that if $F=(f\colon\Xcal\to\Sc, \{\Xcal_i\},\{\Zc_j\}, \bar\delta, \phi)$ is an object of $\Uf_r$, the coefficients of the super Hilbert polynomial of the even line bundle $(\Lcl_F)_{|\Xcal_s}$ are independent of the geometric point $s\in S$, and    the rank of the sheaf $f_\ast(\Lcl_F^{\otimes r})$ at every geometric point of $S$ is a constant pair $M|N$ of integers depending only  on those coefficients \cite[Prop.\! 4.19]{BrHRPo20}. 

Let $\F$ be the presheaf on superschemes defined by letting $\F(\Sc)$ be the set of isomorphism classes of pairs $(F,\sigma)$, where $F=(f\colon\Xcal\to\Sc, \{\Xcal_i\},\{\Zc_j\}, \bar\delta, \phi)$ is an object of $\Uf_r$, and $\sigma$ is an isomorphism $\sigma\colon f_\ast(\Lcl_F^{\otimes r})\iso \Oc_\Sc^{M|N}$. We then have an embedding $\gamma\colon\Xcal \hookrightarrow \Ps_\Sc^{M-1|N}$, and consequently an embedding 
$$
(\phi,\gamma)\colon\Xcal\hookrightarrow \Ycal\times \Ps_\Sc^{M-1|N}
$$
of $\Sc$-superschemes. Let $\Hcal=\SHilb(\Ycal\times \Ps_\Sc^{M-1|N}/\Sc)$ be the Hilbert superscheme of closed subs-uperschemes of $\Ycal\times \Ps_\Sc^{M-1|N}$, flat over $\Sc$ (which exists because $\Ycal$ is projective  \cite[Thm.\! 4.3]{BrHRPo20}), and let
$$
\xymatrix{
\Xcal_{univ}\ar@{^(->}[rr]\ar[rd]^f && (\Ycal\times \Ps_\Sc^{M-1|N})\times_\Sc \Hcal\ar[ld]
\\
& \Hcal &
}
$$
 be the universal closed sub-superscheme. The condition that the fibers of $f$ are Cohen-Macaulay supercurves is open, so that it defines
   an open sub superscheme $\Hcal_1$ of $\Hcal$. Let $f_1\colon \Xcal_1\to \Hcal_1$ be the restriction of $\Xcal_{univ}$ to $\Hcal_1$.  Then, if $\Xcal_{sm}$ is the smooth locus of $f_1$, the complementary $\Hcal_2$ of the diagonals of symmetric product $\Sym^{\nf_{NS}} (\Xcal_{sm}/\Hcal_1)$ parametrizes families of nodal curves with $\nf_{NS}$ Neveu-Schwarz punctures. Let $f_2\colon \Xcal_2\to \Hcal_2$ be the corresponding universal Cohen-Macaulay supercurve with $\nf_{NS}$ Neveu-Schwarz punctures. Notice that by \cite[Prop.\! 5.20]{BrHRPo20} for a proper morphism  there exists 
   the superscheme of relative positive divisors. The proof also works for morphisms that are a composition of an open immersion and a proper morphism; thus, for a proper morphism there exists the superscheme of positive relative divisors that lie in the smooth locus. Using this, we construct a superscheme $\Hcal_3\to \Hcal_2$ that parameterizes Cohen-Macaulay supercurves with $\nf_{NS}$ Neveu-Schwarz punctures and $\nf_{RR}$ Ramond-Ramond punctures, together with a universal Cohen-Macaulay supercurve $(f_3\colon \Xcal_3\to \Hcal_3,\{\Xcal_{3,i}\},\{\Zc_{3,_j}\})$.

By \cite[Prop.\! 3.18]{BrHRPo20}, there exists a superscheme $\Hcal_4\to \Hcal_3$ parametrizing morphisms from $\Omega_{f_3}$ to $\omega_{f_3}(2\mathbf{\Zc_3})$. Let $(f_4\colon \Xcal_4\to \Hcal_4,\{\Xcal_{4,i}\},\{\Zc_{4,_j}\})$ be the induced supercurve with punctures and $\bar\delta\colon \Omega_{f_4}\to\omega_{f_4}(2\Zc_4)$   the universal morphism of $\Oc_{\Xcal_4}$-modules. The locus of the points of $\Hcal_4$ where $\bar\delta$ is surjective and satisfies    condition  (3) of Definition \ref{def:prestablecurve} is an open subsuperscheme $\Hcal_5 $ of $\Hcal_4$, and we have a SUSY curve with punctures $(f_5\colon \Xcal_5\to \Hcal_5,\{\Xcal_{5,i}\},\{\Zc_{5,_j}\},\bar\delta)$.

The rest of the proof is   an adaptation of the final part of the proof of \cite[13.3.10]{Ol16}. 
We have a map $\phi_5\colon \Xcal_5 \to \Ycal$ and we can then consider the even line bundle  $\Lcl_F$ as in Definition \ref{def:stablemap2} for $F=(f_5\colon \Xcal_5\to \Hcal_5,\{\Xcal_{5,i}\},\{\Zc_{5,_j}\},\bar\delta,\phi_5)$. Let us denote by $\Nc$ the even line bundle on $\Xcal_5$ obtained as the pullback of the sheaf $\Oc_{\Ps_\Sc^{M-1|N}}(1)$ on $\Ps_\Sc^{M-1|N}$. Again by \cite[Prop.\! 3.18]{BrHRPo20}, there exists a superscheme $\Hcal_6\to\Hcal_5$ parameterizing isomorphisms of $\Lcl_F^{\otimes r}$ with $\Nc$. Then $\Hcal_6$ represents the functor that to a superscheme $\Sc$ associates  tripes consisting of a prestable SUSY curve with punctures  $(f\colon\Xcal\to\Sc, \{\Xcal_i\},\{\Zc_j\}, \bar\delta)$, an embedding $(\phi,\gamma)\colon\Xcal\hookrightarrow \Ycal\times \Ps_\Sc^{M-1|N}$ and an isomorphism $\tilde\sigma\colon \Lcl_F^{\otimes r} \iso \gamma^\ast\Oc_{\Ps_\Sc^{M-1|N}}(1)$. The condition that the composition map
$$
\Oc_\Sc^{M|N}\to f_\ast\gamma^\ast \Oc_{\Ps_\Sc^{M-1|N}}(1) \xrightarrow{\tilde\sigma} f_\ast(\Lcl_F^{\otimes r})
$$
is an isomorphism is also an open condition. We then have an open sub-superscheme $\Ucal_r\hookrightarrow \Hcal_6$ that represents the functor $\F$.

There is an action of the linear supergroup $\mathbb{G}L_{M|N}(\Oc_{\Sc})$ on $\Ucal_r$, induced by the action on $\F$, which  changes the choice of the isomorphism $\sigma$, and  there is a stack isomorphism
$$
\Uf_r\iso [\Ucal_r/\mathbb{G}L_{M|N}(\Oc_{\Sc})]\,.
$$
This finishes the proof of the algebraicity of $\Sf\Mf(\Ycal,\beta)$ by using \cite[Example 3.15]{BHRSS25}.

The diagonal  morphism is superschematic due to 
Lemma \ref{lem:represdiag}.
To prove that the diagonal of $\Sf\Mf(\Ycal,\beta)$ is separated we use  again Lemma \ref{lem:represdiag}, in this case its proof. There we have proved that for every stable supermap $F:=(f\colon\Xcal\to\Sc, \{\Xcal_i\},\{\Zc_j\}, \bar\delta, \phi)$    of genus $g$ over a superscheme $\Sc$, the fiber product $\Sf\Mf(\Ycal,\beta)\times_{\Delta,\Sf\Mf(\Ycal,\beta)\times\Sf\Mf(\Ycal,\beta),(F,F)}\Sc$ is representable by a closed superscheme $\spacecal{P}$ of a superscheme $\Vc$ of automorphisms of $\Xcal$. Now $\Vc$ is separated because it is an open sub-superscheme of a Hilbert superscheme and the latter is separated by \cite[Thm.\! 4.3]{BrHRPo20}. Then $\Pc$ is separated, and we have finished.
\end{proof}

\begin{rem} This also implies that the superstack of stable SUSY curves constructed in \cite{FKPpubl} has a superschematic and separated diagonal.
\end{rem}
 
\subsection{The superstack of stable supermaps is Deligne-Mumford} 
The aim of this subsection is to prove that the superstack of stable supermaps is Deligne-Mumford. The proof is based on the properties of the diagonal, namely we are going to prove that the algebraic space  of automorphisms of a geometric point of the algebraic superstack of stable supermaps is a finite and reduced algebraic supergroup, and then use  \cite[Prop.\! 3.68]{BHRSS25}. For that, we need to study the infinitesimal automorphisms of a stable supermap.

Let $k$ be an algebraically closed field.
We start by fixing some notation following \cite{FKPpubl}.
Let $(\Xcal, \{x_i\},\{\Zc_j\},\bar\delta)$ be a prestable SUSY curve over $k$ with Neveu-Schwarz  punctures $ \{x_i\}$ and Ramond-Ramond punctures $\{z_j\}$. Then one has $\Oc_\Xcal=\Oc_X\oplus\Pi\Lcl$ for a pure one-dimensional  $\Oc_X$-module  $\Lcl$ and
an isomorphism of $\Oc_X$-modules $\bar\delta \colon
\Lcl\iso \Homsh_{\Oc_X}(\Lcl, \omega_X(Z))$ ($\Zc=\sum_j \Zc_j$).  We can construct a smooth SUSY-curve with punctures, $(\widetilde\Xcal,  \{\tilde x_h\},\{\widetilde \Zc_j\},\bar{\delta'})$, the \emph{pointed normalization} of $\Xcal$, by iterating for every node of $X$ the construction given in \cite[Sec.\! 2.6]{FKPpubl}. Then:
\begin{itemize}
\item the bosonic reduction of $\widetilde\Xcal$ if the normalization $\widetilde X$ of $X$;
\item $\Oc_{\widetilde\Xcal}=\Oc_{\widetilde X}\oplus\Pi \widetilde\Lcl$, where $\widetilde\Lcl=\rho^\ast\Lcl/\{\text{torsion}\}$, $\rho\colon \widetilde X \to X$ being the desingularization morphism;
\item $\widetilde\Xcal$ inherits all the  Ramond-Ramond punctures $\{x_i\}$ of $\Xcal$, and for every  Ramond-Ramond node of $X$, that is, a node where $\Lcl$ is locally free, we add the superdivisors corresponding to the two points of the fibre of $\rho$ (\cite[Definition 2.13]{FKPpubl}) over it as  Ramond-Ramond punctures,
\item $\widetilde\Xcal$ inherits all the  Neveu-Schwarz  punctures $\{\Zc_j\}$ of $\Xcal$, and for every  Neveu-Schwarz  node of $X$, that is, a node where $\Lcl$ is not locally free, we add the two points of the fibre of $\rho$ over it as  Neveu-Schwarz punctures.
\end{itemize}
Then $\rho$ induces a morphism of SUSY curves with punctures, which we still denote by 
$$
\rho\colon (\widetilde\Xcal,  \{\tilde x_h\},\{\widetilde \Zc_\ell\},\bar{\delta'})\to (\Xcal, \{x_i\},\{\Zc_j\},\bar\delta)\,.
$$
One can see that \emph{if $(\Xcal, \{x_i\},\{\Zc_j\},\bar\delta)$ is stable, so is its pointed normalization $(\widetilde\Xcal,  \{\tilde x_h\},\{\widetilde \Zc_\ell\},\bar{\delta'})$}.

We can extend this construction to stable supermaps. For that, consider a projective superscheme $\Ycal$  and a rational equivalence class of supercycles $\beta\in A_1(\Ycal$). 
If $F=(\Xcal, \{x_i\},\{\Zc_j\},\bar\delta, \phi)$ is a stable supermap to $\Ycal$ of class $\beta$, so that $\phi\colon \Xcal \to \Ycal$ is a morphism with $\phi_\ast[\Xcal]=\beta$, the morphism $\widetilde\phi=\phi\circ\rho \colon\widetilde\Xcal\to \Ycal$ gives rise to a family $\widetilde F=(\widetilde\Xcal, \{\tilde x_h\},\{\widetilde \Zc_\ell\},\bar\delta', \widetilde\phi)$. From the definition of the punctures of $\widetilde\Xcal$ one sees that $\widetilde F$ \emph{is a stable supermap  to $\Ycal$ of class $\beta$ from the pointed normalization}.


Following \cite{FKPpubl}, we denote by $\Ac_{(\Xcal, \{x_i\},\{\Zc_j\},\bar\delta)}$ the sheaf of infinitesimal automorphisms of a prestable SUSY curve with punctures, and by   $\Ac_{(\Xcal, \{x_i\},\{\Zc_j\},\bar\delta,\phi)}$ the sheaf of infinitesimal automorphisms of a superstable map to $\Ycal$. There is a natural injective morphism
\begin{equation}\label{eq:infautomor0}
\varpi\colon\Ac_{(\Xcal,\{x_i\},\{\Zc_j\},\bar\delta,\phi)}\hookrightarrow \Ac_{(\Xcal,\{x_i\},\{\Zc_j\},\bar\delta)}\,.
\end{equation}
We have also morphisms 
\begin{equation}\label{eq:infautomor}
\Phi\colon\Ac_{(\Xcal, \{x_i\},\{\Zc_j\},\bar\delta)} \to \rho_\ast\Ac_{(\widetilde\Xcal, \{\tilde x_h\},\{\widetilde\Zc_\ell\},\bar\delta')}\,,\quad
\Psi\colon\Ac_{(\Xcal, \{x_i\},\{\Zc_j\},\bar\delta,\phi)} \to \rho_\ast\Ac_{(\widetilde\Xcal, \{\tilde x_h\},\{\widetilde\Zc_\ell\},\bar\delta',\phi)}\,.
\end{equation}

\begin{lemma}\label{lem:infautomor} The morphism $\Psi$ in  Equation \eqref{eq:infautomor} is injective.
\end{lemma}
\begin{proof} We have a diagram with exact rows
$$\xymatrix{  0\ar[r] &\Ac_{(\Xcal, \{x_i\},\{\Zc_j\},\bar\delta,\phi)} \ar@^{(->}[r]^\varpi \ar[d]^\Psi& \Ac_{(\Xcal, \{x_i\},\{\Zc_j\},\bar\delta)}\ar[d]^\Phi
\\
0\ar[r] &\rho_\ast\Ac_{(\widetilde\Xcal, \{\tilde x_h\}, \{\widetilde\Zc_\ell\},\bar\delta',\phi)}\ar@^{(->}[r]^{\rho_\ast\varpi} & \rho_\ast\Ac_{(\widetilde\Xcal, \{\tilde x_h\},\{\widetilde\Zc_\ell\},\bar\delta')}
}
$$
By \cite[Prop.\! 3.11]{FKPpubl} (whose part (i) also holds for prestable curves), $\Phi$ is injective. Then, $\Psi$ is injective as well.
\end{proof}

%

Before establishing our main result on the infinitesimal automorphism of a stable supermap, we study some preliminary results on its bosonic reduction in the particular case we shall need it.

First notice that  if $F=(\Xcal, \{x_i\},\{\Zc_j\},\bar\delta, \phi)$ is a stable supermap to $\Ycal$ of class $\beta$, the class $\beta$ is of the form  $\beta=(1-\Pi)\beta_0$ for a cycle $\beta_0$ on the bosonic reduction $Y$ of $\Ycal$; moreover   the condition $\phi_\ast[\Xcal]=\beta$  is equivalent to $(\phi_{bos})_\ast[X_s]=\beta_0$.

Consider now the case when $\Ycal=\Ps^{r|s}$ is a projective superspace. Then $Y=\Ps^r$ and $\beta_0=d H^{r-1}$ where $H$ is the hyperplane class. Assume also that $X$ is a smooth and irreducible curve that is not contracted by $\phi_{bos}$. Then, the schematic image $C:=\Im\phi$ is an irreducible curve and $\phi_{bos}\colon X \to C$ is a finite morphism. If $m$ is its degree and $q$ is the degree of the curve $C$ in $\Ps^r$, we have $d=mq$.

A standard argument of projective geometry yields:
\begin{lemma}\label{lem:addingpoints} In the above situation,  there is an open subset $\U$ of the dual projective space $(\Ps^r)^\ast$ whose closed points correspond to hyperplanes $H\subset \Ps^r$ fulfilling these conditions:
\begin{enumerate}
\item $H$ does not contain any of the following points:
\begin{enumerate}
\item any of the images $\phi_{bos}(x_i)$, $\phi_{bos}(z_j)$ of the special points $\{x_i\}$, $\{z_j\}$;
\item  the singular points of $C$;
\item the ramification values of $\phi$ over the smooth locus of $ C$.
\end{enumerate}
\item $H\cap \Im\Phi$ consists of $q$ different points $\{y_1,\dots,y_q\}$.
\end{enumerate}
Note that $\U$ is nonempty as $k$ is infinite. Moreover, proceeding iteratively, one can find hyperplanes $H_1$, $H_2$ and $H_3$ in $\Ps^r$ fulfilling 1 and 2    that   not meet pairwise along $C$.
\end{lemma}

\begin{prop}\label{prop:infautomor} If $F=(\Xcal, \{x_i\},\{\Zc_j\},\bar\delta,\phi)$ is a stable supermap to a projective superscheme $\Ycal$, the space of  its  global infinitesimal automorphisms is zero, $H^0(X,\Ac_{(\Xcal, \{x_i\},\{\Zc_j\},\bar\delta,\phi)} )=0$
\end{prop}
\begin{proof} By Lemma \ref{lem:infautomor}, we can assume that $\Xcal$ is smooth. Moreover, we can   assume that it is connected. One has
\begin{equation}\label{eq:stablereduction}
0\to H^0(X,\Ac_{(\Xcal, \{x_i\},\{\Zc_j\},\bar\delta,\phi)})\xrightarrow{H^0(\varpi)} H^0(X,\Ac_{(\Xcal, \{x_i\},\{\Zc_j\},\bar\delta)})\,.
\end{equation}
If $X$ is contracted by $\phi\colon \Xcal \to \Ycal$, then $\varpi\colon\Ac_{(\Xcal, \{x_i\},\{\Zc_j\},\bar\delta,\phi)}\hookrightarrow \Ac_{(\Xcal, \{x_i\},\{\Zc_j\},\bar\delta)}$ is an isomorphism and 
the punctured SUSY curve $(\Xcal, \{x_i\},\{\Zc_j\},\bar\delta)$ is stable. Thus, 
$$
H^0(X,\Ac_{(\Xcal, \{x_i\},\{\Zc_j\},\bar\delta,\phi)})=H^0(X,\Ac_{(\Xcal, \{x_i\},\{\Zc_j\},\bar\delta)} )=0
$$ by \cite[Prop.\! 3.11]{FKPpubl}.

We are left with the case where $\Xcal$ is not contracted. We can assume that $\Ycal=\Ps^{r|s}$. 

Take the hyperplanes $H_i$ in $\Ps^r$ given by  Lemma \ref{lem:addingpoints}, and let $\Hc_i$ be the superdivisors of $\Ps^{r|s}$ obtained by pull-back under the projection $\Ps^{r|s}\to \Ps^r$. 

The superdivisor $\sum_{\alpha=1,2,3}\phi^\ast\Hc_\alpha$ of $\Xcal$ is the sum of $3d=3m q$ superdivisors $\Wc'_{\alpha h}$  supported on $3d$ pairwise disjoint points of $X$. If  $x'_{\alpha h}$ are the sections of $\Xcal\to\Spec k$ corresponding to them by \cite[Def.\! 2.13]{FKPpubl}, we can add $\{x'_{\alpha h}\}$ as Neveu-Schwarz punctures,   getting a new prestable punctured SUSY curve $(\Xcal,\{x_i, x'_{\alpha h}\},\{\Zc_{j}\},\bar\delta)$. 
We now have:
\begin{enumerate}
\item
since the supermap $F$ is stable, the line bundle $\Lcl_F=\omega_\Xcal^2(\sum \Wc_i+\sum\Zc_j)\otimes \phi^\ast\Oc_{\Ps^{r|s}}(1)^{\otimes 3}$ is ample (Definition \ref{def:stablemap2}). We can write $\Lcl_F$ as
$$
\Lcl_F= \omega_\Xcal^2(\sum \Wc_i+\sum\Zc_j)\otimes \Oc_\Xcal(\phi^{-1}(\Hc_1+\Hc_2+\Hc_3))\simeq \omega_\Xcal^2\left(\sum (\Wc_i+\Wc'_{\alpha h})+\sum\Zc_j\right)
$$
so that the punctured SUSY curve $(\Xcal,\{x_i, x'_{\alpha h}\},\{\Zc_j\},\bar\delta)$ is stable (Definition \ref {def:stablecurve}). Thus, 
\begin{equation}\label{eq:stablereduction2} 
H^0(X,\Ac_{(\Xcal,\{x_i, x'_{\alpha h}\},\{\Zc_{j}\},\bar\delta)})=0
\end{equation}
by \cite[Prop.\! 3.11]{FKPpubl}.
\item 
Let us now consider $\Sc=\SSpec k[\epsilon,\theta]$, the superscheme of the   dual supernumbers, that is, $\epsilon$ is even, $\theta$ is odd, and $\epsilon^2=\theta^2=\epsilon\theta=0$. We have a base-changed prestable  SUSY curve $(\Xcal_\Sc\to\Sc,\{x_{i\Sc}\},\{\Zc_{j\Sc}\},\bar\delta_\Sc)$ and a stable supermap $F_\Sc=(\Xcal_\Sc\to\Sc,\{x_{i\Sc}\},\{\Zc_{j\Sc}\},\bar\delta_\Sc,\phi_\Sc)$, whose restrictions to the fibre of the closed point of $\Sc$ are the original ones. 

Take as above the hyperplanes $H_\alpha$ in $\Ps^r$ given by  Lemma \ref{lem:addingpoints}, and let $\Hc_{\alpha\Sc}$ be the superdivisors of $\Ps_\Sc^{r|s}$ obtained by pull-back under the projection $\Ps_\Sc^{r|s}\to \Ps_\Sc^r\to \Ps^r$. Proceeding as before, we get a stable SUSY curve $(\Xcal_\Sc\to\Sc,\{x_{i\Sc},x'_{\alpha h\Sc}\},\{\Zc_{j\Sc}\},\bar\delta_\Sc)$ over $\Sc$. 
Since the automorphisms of the supermap $F_\Sc$ over $\Sc$ preserve the fibres of $\phi_\Sc$, they also preserve the new Neveu-Schwarz punctures and yield automorphisms of the SUSY curve $(\Xcal_\Sc\to\Sc,\{x_{i\Sc},x'_{\alpha h\Sc}\},\{\Zc_{j\Sc}\},\bar\delta_\Sc)$ over $\Sc$ . In particular, we have an immersion
$$
\Ac_F
\hookrightarrow 
\Ac_{(\Xcal,\{x_i, x'_{\alpha h}\},\{\Zc_{j}\},\bar\delta)} 
$$
and one finishes by Equation \eqref{eq:stablereduction2}.
\end{enumerate}
\end{proof}

\begin{thm}\label{thm:DMstack}  
The superstack
$\Sf\Mf_{g,\nf_{NS},\nf_{RR}}(\Ycal,\beta) \xrightarrow{p} \Sf$ of the stable supermaps is a Deligne-Mumford superstack.
\end{thm}
\begin{proof} Let us write for simplicity $\Sf\Mf$ instead of $\Sf\Mf_{g,\nf_{NS},\nf_{RR}}(\Ycal,\beta)$. By Theorem \ref{thm:algstack}  it is an algebraic superstack with superschematic diagonal. Thus, for every geometric point of $\Sf\Mf$, which is given by a stable supermap $F=(\Xcal,\Wc,\Zc,\bar\delta,\phi)$ over an algebraically closed field $k$, the algebraic superspace $\mathfrak{A}ut_{\Sf\Mf(k)}(F)$ of the automorphisms of $F$ is a superscheme.  The tangent space to  $\mathfrak{A}ut_{\Sf\Mf(k)}(F)$ at the identity  is the space of   global infinitesimal automorphisms of $F$, and this is zero by Proposition \ref{prop:infautomor}. Then $\mathfrak{A}ut_{\Sf\Mf(k)}(F)$ is a discrete and reduced group scheme so that $\Sf\Mf$ is a Deligne-Mumford superstack by  \cite[Prop.\! 3.68]{BHRSS25}.
\end{proof}

\section{Bosonic reduction} \label{red}

Since the properness of a superscheme or superstack is equivalent to the properness of the underlying bosonic reduction \cite[Prop.\! 3.45]{BHRSS25},
it is important to compute the latter for the superstack of stable supermaps. In the case of stable SUSY curves, the bosonic reduction is the stack of stable curves with a spin structure. This stack has a finite morphism onto the stack of stable curves, and therefore it is proper. 

One might think that the bosonic reduction of the superstack $ \Sf\Mf(\Ycal,\beta)$ of stable supermaps into a superscheme $\Ycal$ is the stack  $\Mf^{spin}_{g}(Y,\beta_0)$ of maps from curves with a spin structure to the bosonic reduction $Y$ of $\Ycal$, but we shall see that this is not always the case (for simplicity  in this preliminary discussion we omit references to the punctures). The bosonic reduction of $ \Sf\Mf(\Ycal,\beta)$ has a morphism to $\Mf^{spin}_{g}(Y,\beta_0)$ whose fibers are linear schemes, so that $ \Sf\Mf(\Ycal,\beta)$ is not proper unless these linear schemes reduce to a point. On the other hand, if the target $\Ycal$ is an ordinary scheme $Y$, the bosonic reduction of $ \Sf\Mf(Y,\beta_0)$
is exactly $\Mf^{spin}_{g}(Y,\beta_0)$, and $ \Sf\Mf(Y,\beta_0)$ is proper.

In the next subsection we study the moduli stack of stable spin maps; then in Subsection \ref{bosonic} we shall compute the bosonic reduction of $ \Sf\Mf(\Ycal,\beta)$.

\subsection{The stack of spin maps}

We start by studying the stack
of stable spin maps, that is, stable maps from curves with punctures equipped with a spin structure. So, in this subsection we consider only schemes and no superschemes. Note that in spite of a spin curve being a completely bosonic object (no ``super'' structure) it makes sense to talk of Ramond-Ramond punctures. The stack of stable spin maps was already studied in \cite{JaKiVa05,MoZh24}.

The precise definitions are the following:
\begin{defin}\label{def:spin} A prestable spin curve of genus $g$ with $\nf_{NS}$ Neveu-Schwarz punctures and $\nf_{RR}$ Ramond-Ramond punctures over a scheme $S$ is defined by the following data:
\begin{enumerate}
\item a proper Cohen-Macaulay scheme morphism $f\colon X \to S$ whose fibers are connected prestable (i.e. nodal) curves with arithmetic genus $g$. Then, $f\colon X \to S$ is automatically a Gorenstein morphism; 
\item a family of $\nf_{NS}$ sections $\sigma_i\colon S \to X$ of $f$ whose images $W_i$ are contained in the smooth locus of $f$;
\item a family $\{Z_j\}$ of $\nf_{RR}$ divisors   of relative degree 1 over $S$ contained in the smooth locus of $f$;
\item a coherent sheaf $\Lcl$ on $X$, flat and relatively Cohen-Macaulay over $S$,  generically of rank one;
\item an isomorphism $\varpi\colon\Lcl\iso\Homsh_{\Oc_X}(\Lcl,\omega_{X/S}(Z))$, where $Z=\sum_j Z_j$.
\end{enumerate}
\end{defin}

Let $Y$ be a projective scheme, and   fix an ample line bundle $\Lcl_Y$ and a 1-cycle $\beta_0$ on $Y$.
 
\begin{defin}\label{def:spinmap} 
A stable spin map $F$ into $Y$, of genus $g$ and  class $\beta_0$, with $\nf_{NS}$ Neveu-Schwarz punctures and $\nf_{RR}$  Ramond-Ramond punctures over a scheme $S$ is given by the following data:
\begin{enumerate}
\item A prestable spin curve of genus $g$ with $\nf_{NS}$ Neveu-Schwarz punctures and $\nf_{RR}$  Ramond-Ramond punctures $(f\colon X\to S,\{W_i\}, \{Z_j\},  \Lcl,\varpi)$ over $S$ (Definition \ref{def:spin}).
\item A scheme morphism $\psi\colon X\to Y$ such that $\psi_\ast[X_s]=\beta_0$ for every closed point $s\in S$.
\end{enumerate}
Moreover, we assume that the line bundle
$$
\Lcl_F:=\omega_{X/S}(W+Z)\otimes \psi^\ast \Lcl_Y^{\otimes 3}\,,
$$ 
where $W=\sum_i W_i$, is strongly relatively ample with respect to $f\colon X\to S$.
\end{defin}

Proceeding in the usual way we can define the category fibered in groupois over the \'etale site $Sch_{et}$ of schemes defined by the stable spin maps  $F=(f\colon X\to S,\{W_i\}, \{Z_j\},  \Lcl,\varpi, \psi)$ into $Y$, of  genus $g$ and  class $\beta_0$, with $\nf_{NS}$ Neveu-Schwarz punctures and $\nf_{RR}$  Ramond-Ramond punctures. We denote it by $\Mf^{spin}_{g,\nf_{NS},\nf_{RR}}(Y,\beta_0)$. 
 One has \cite[Thm.\! 2.3.2]{JaKiVa05}:
\begin{prop}\label{prop:spinstack}
$\Mf^{spin}_{g,\nf_{NS},\nf_{RR}}(Y,\beta_0)$ is a proper Deligne-Mumford stack.
\qed
\end{prop}

\subsection{The bosonic reduction of the superstack $\Sf\Mf(\Ycal,\beta)$}
\label{bosonic}
In this section we describe the relation between   the bosonic reduction of the superstack of stable supermaps $\Sf\Mf(\Ycal,\beta)$ and the stack of stable spin maps.  

 The bosonic reduction of a superstack  was defined in  \cite{CodViv17}, see also \cite[Def.\! 2.25]{BHRSS25}. We apply that construction
to characterize the  bosonic reduction 
 $\Sf\Mf(\Ycal,\beta)_{bos}$ of the   
 Deligne-Mumford superstack $\Sf\Mf(\Ycal,\beta)$.  By Theorem \ref{thm:DMstack}  and \cite[Prop.\! 3.20]{BHRSS25}, this  is a Deligne-Mumford stack
 over the category of ordinary schemes, whose objects over a scheme $S$ are the stable  supermaps
$F=(f\colon\Xcal\to S, \{\Xcal_i\},\{\Zc_j\}, \bar\delta, \phi)$ into $\Ycal$ of class $\beta$.

Since  $\beta$   is the cycle of the image in $\Ycal$ of  a superscheme of dimension $1|1$, it  is of the form $\beta=(1-\Pi)\beta_0$ for a cycle $\beta_0$ on the bosonic reduction $Y$ of $\Ycal$; moreover   the condition $\phi_\ast[\Xcal_s]=\beta$ for a geometric point $s\in S$  is equivalent to $(\phi_{bos})_\ast[X_s]=\beta_0$.
 
\begin{lemma}\label{lem:stmapsoversch} The data of a stable map $F=(f\colon\Xcal\to S, \{\Xcal_i\},\{\Zc_j\}, \bar\delta, \phi)$ into $\Ycal$ of class $\beta$ over an ordinary scheme $S$ is equivalent to the following data:
\begin{enumerate}
\item
A stable spin map $F_{spin}:=(f_{bos}\colon X \to S, \{X_i\},\{Z_j\}, \varpi, \phi_{bos})$ into $Y$ of class $\beta_0$;
\item
A morphism $\lambda_F\colon\phi_{bos}^\ast \F_\Ycal \to \Lcl$ of $\Oc_X$-modules, where  $\Lcl$ is defined by $\Oc_\Xcal=\Oc_X\oplus\Pi \Lcl$.
\end{enumerate}
\end{lemma}  
\begin{proof}
We have $\Oc_\Xcal=\Oc_X\oplus \Pi\Lcl$ for a relatively Cohen-Macaulay sheaf $\Lcl$, generically of rank one. It is known that a prestable SUSY curve over an ordinary scheme $(f\colon\Xcal\to S, \{\Xcal_i\},\{\Zc_j\}, \bar\delta)$ is equivalent to a prestable spin curve $(f\colon X\to S, \{W_i\},\{Z_j\}, \varpi)$ (see \cite{FKPpubl} or \cite{BrHR21}). Moreover, using \cite[Thm.\! 4.1]{FKPpubl}, we see that the bosonic reduction of the $\Oc_\Xcal$ module $\Lcl_F=\omega_{\Xcal/S}^2\big(\sum_{1\le i\le \nf_{NS}} \Wc_i + \sum_{1\le j\le \nf_{RR}} \Zc_j\big)\otimes \phi^\ast \Lcl_\Ycal^{\otimes 3}$ of Definition \ref{def:stablemap2} is the sheaf
$$
\Lcl_{F_{spin}}:=\omega_{X/S}(W+Z)\otimes \phi_{bos}^\ast \Lcl_Y^{\otimes 3}\,,
$$
where $\Lcl_Y=(\Lcl_\Ycal)_{bos}$, $W=\sum_{1\le i\le \nf_{NS}}W_i$ and $Z=\sum_{1\le j\le \nf_{RR}} Z_j$.
By \cite[Prop.\! A.2]{FKPpubl} the ampleness of $\Lcl_{F_{spin}}$ is equivalent to the ampleness of $\Lcl_F$. Then, $F$ is a stable supermap if and only if $(f_{bos}\colon X \to S, \{X_i\},\{Z_j\}, \varpi, \phi_{bos})$ is a stable spin map.

We now analize the relationship between $\phi\colon \Xcal \to \Ycal$ and its bosonic reduction $\phi_{bos}\colon X \to Y$. The morphism $\phi\colon \Xcal \to \Ycal$ is given by a superring sheaf morphism $\Oc_\Ycal \to (\phi_{bos})_\ast\Oc_\Xcal = (\phi_{bos})_\ast\Oc_X\oplus (\phi_{bos})_\ast(\Pi \Lcl)$ extending the morphism $\Oc_Y \to (\phi_{bos})_\ast\Oc_X$ induced by $\phi_{bos}$. Since $(\Pi\Lcl)^2=0$ in $\Oc_\Xcal$,  the ideal $\Jc_\Ycal^2$ is sent  to zero,  and then $\phi$ is characterized by $\phi_{bos}$ together with a sheaf morphism $\F_{\Ycal}=\Jc_\Ycal /\Jc_\Ycal^2 \to (\phi_{bos})_\ast(\Pi \Lcl)$. This is the same as a morphism $\phi_{bos}^\ast \F_\Ycal \to \Lcl$, so that the proof is finished.
\end{proof}

Lemma \ref{lem:stmapsoversch}  allow us to describe a stable supercurve $F$ over an ordinary scheme $S$ as a pair $F\equiv (F_{spin},\lambda_F)$, where $F_{spin}$ is a stable spin map and $\lambda_F\colon\phi_{bos}^\ast \F_\Ycal \to \Lcl$ is a morphism of $\Oc_X$-modules. 
 It follows that we have a forgetful stack morphism between Deligne-Mumford stacks
$$
\Ff\colon \Sf\Mf_{g,\nf_{NS},\nf_{RR}}(\Ycal,\beta)_{bos} \to \Mf^{spin}_{g,\nf_{NS},\nf_{RR}}(Y,\beta_0)
$$
from the bosonic reduction of the superstack of stable supermaps to the stack of stable spin maps, given by $F = ( (F_{spin},\lambda_F)\mapsto F_{spin})$. 
\begin{prop}\label{prop:bosredDM} The forgetful stack morphism $\Ff$ is schematic and affine. 
\end{prop}
\begin{proof} Let $S\to \Mf^{spin}_{g,\nf_{NS},\nf_{RR}}(Y,\beta_0)$ be a stack morphism from a scheme $S$ given by a stable spin map  $F_{spin}=(f\colon X\to S,\{W_i\}, \{Z_j\},  \Lcl,\varpi, \psi)$. By Lemma \ref{lem:stmapsoversch} the fiber of $\Ff$ over $F_{spin}$ is the stack whose objects over a $S$-scheme $T$ are the 
morphisms $\psi_T^\ast (\F_\Ycal)_T \to \Lcl_T$ of $\Oc_{X_T}$-modules. By \cite[Cor.\! 7.7.8]{EGAIII-II} (see also \cite[Thm.\! 5.8]{Ni07}), this stack is representable by a linear scheme $V(\Qc)=\Spec( \Sym (\Qc))\to X$, where $\Qc$ is a coherent sheaf on $X$. 
\end{proof}

Now $\Mf^{spin}_{g,\nf_{NS},\nf_{RR}}(Y,\beta_0)$ is proper by Proposition \ref{prop:spinstack},  but the   stack $\Sf\Mf_{g,\nf_{NS},\nf_{RR}}(\Ycal,\beta)_{bos}$ may fail to be proper: it is proper if and only if $\Ff$ is an isomorphism.  As a consequence, the superstack $\Sf\Mf_{g,\nf_{NS},\nf_{RR}}(\Ycal,\beta)$ of stable supermaps may fail to be proper. Note, on the other hand, that when the target $\Ycal$ is bosonic, i.e., it is an ordinary scheme $Y$,  the stack $\Sf\Mf_{g,\nf_{NS},\nf_{RR}}(Y,\beta)_{bos}$ coincides with $\Mf^{spin}_{g,\nf_{NS},\nf_{RR}}(Y,\beta_0)$, and therefore is always proper.

We shall return to this point in Section \ref{dim} (Example \ref{notprop}).
%
%
%
%

 \section{Characteristic classes and super Grothendieck-Riemann-Roch}
 \label{char}
 
 To compute the virtual dimension of the stack of stable supermaps one needs a super version of the 
 Grothendieck-Riemann-Roch theorem. This has been proved by Manin, Penkov and Voronov in \cite{VMP}. In this section we make a review of the constructions leading to that theorem; on the one hand we shall streamline the exposition with respect to \cite{VMP}, including only the results that we shall need, but on the other hand we shall also add some clarificatory details and give some proofs that in \cite{VMP} were not included, or just sketched.  
  
 In the following subsections we recall from  \cite{VMP}  a version of K-theory for superschemes, which is a not entirely trivial extension to superschemes of the approach to ordinary K-theory of \cite{Ma69}, and an approach to characteristic classes for super vector bundles; we give some details of a suitable splitting principle; we introduce a Todd character, and finally state  the super Grothendieck-Riemann-Roch theorem and reproduce its proof.

 \subsection{Super K-theory}

  Let $\Xcal = (X,\Oc_\Xcal)$    be a   noetherian    superscheme over an algebraically closed field.
   $K_\bullet^S(\Xcal)$ and $K^\bullet_S(\Xcal)$ will denote the K-theory groups of coherent and locally free finitely generated graded $\Oc_\Xcal$-modules, respectively, while $K_\bullet^S(X)$ and $K^\bullet_S(X)$   will denote the K-theory groups of $\Z_2$-graded  coherent and locally free finitely generated graded $\Oc_X$-modules. We shall assume that the natural morphism $K^\bullet(X) \to K_\bullet(X)$
   between the ordinary K-theory groups of the bosonic reduction $X$ is an isomorphism (this happens for instance when $X$ is a smooth variety). 
   
  $\cl_\bullet $ and $\cl^\bullet$  will denote the operation of taking the classes in  $K_\bullet^S(\Xcal)$ and $K^\bullet_S(\Xcal)$, respectively, while
  we shall denote by $\cl$ the operation of taking the class   in both isomorphic rings $K_\bullet^S(X)$ and $K^\bullet_S(X)$. The group $K^\bullet_S(\Xcal)$ is a commutative ring with respect to the tensor product operation. It contains distinguished elements $ 1 = \cl^\bullet(\Oc_\Xcal)$ and $\Pi = \cl^\bullet(\Pi\Oc_\Xcal)$. Note that $\Pi^2=1$.
  
  \begin{defin}  If $f\colon\Xcal\to \Ycal$ is a proper morphism, the morphism $f_!^S \colon K_\bullet^S(\Xcal)\to  K_\bullet^S(\Ycal)$ is defined by 
  $$ f_!^S(\cl_\bullet \Fc)  = \sum_i (-1)^i \cl_\bullet (R^if_\ast \Fc).$$
  \end{defin}
  
  \begin{defin}  We shall denote by $j\colon K^\bullet_S(\Xcal) \to K^\bullet_S(X)$ the morphism which takes the ``twisted graded module''  $ j(\cl^\bullet(\Ec) )= \cl(\tgr \Ec)$, where
  $$ \tgr \Ec =  \bigoplus_i \Pi^i \Jc^i\Ec/\Jc^{i+1}\Ec. $$ Then we set
  $$ KS(\Xcal) = \Im  j \subset K^\bullet_S(X).$$
  \end{defin}
  
We shall denote by $\cl^S$ the composition $j\circ\cl^\bullet $.

  \begin{prop} If $i\colon X \to \Xcal$ is the canonical embedding, the morphism $i_!^S \colon K_\bullet^S(X)\to  K_\bullet^S(\Xcal)$ is an isomorphism.
\label{isoKS}  \end{prop}

  \begin{proof}  An inverse for $i_!^S$ is provided by the morphism which sends $\cl_\bullet(\M)$ to the class of the graded object $\gr\M =  \bigoplus_i  \Jc^i\M/\Jc^{i+1}\M$ of $\M$. 
  \end{proof}
  
 Let $N_\Xcal$ be the normal sheaf of the embedding $i\colon X \to \Xcal$. We simply write $N$ for $N_\Xcal$ when no confusion may arise. Let $$\sigma_1(N^\ast) =\sum_i\cl(\Pi^i \operatorname{Sym}^i(N^\ast)) =\cl(\tgr\Oc_\Xcal) = j(1)
\in K_\bullet^S(X)\,.$$

  \begin{lemma}\label{lem:KS} $j(\cl^\bullet \Ec) = \cl(\Ec\redu) \cdot \sigma_1(N^\ast)$,  where $\Ec\redu=i^\ast\Ec$ is the bosonic reduction of $\Ec$, and $KS(\Xcal) = \sigma_1(N^\ast) \cdot  K^\bullet_S(X)$ if $\Xcal$ is projected.
  \end{lemma}
  
  \begin{rem} \label{sigmainvert}
 $\sigma_1(N^\ast)$ is invertible in  $K_\bullet^S(X)\otimes \Q$. Note that the usual algebraic K-theory $K_\bullet(X)$ of a variety $X$ over $k$ has a ring structure given by
$$ \cl(\Ec_1) \cdot \cl(\Ec_2) = \sum_{j=0}^{\dim X} (-1)^j \cl (\mathcal Tor_j^{\Oc_X}(\Ec_1,\Ec_2)).$$
 The Chern character establishes an isomorphism of 
 $K_\bullet(X)\otimes \Q$ with the rational Chow ring $A^\bullet(X)\otimes \Q$, where the ring structure is given by the intersection product. As a result any element in $K_\bullet(X)$ which is of the form
 $$ 1 + \text{terms that in $A^\bullet(X)\otimes \Q$ have positive degree} $$
 is invertible.
 Since $$ K^S_\bullet(X) = K_\bullet(X) \oplus\Pi K_\bullet(X)\simeq K_\bullet(X)\otimes_\Z \Z[\Pi]\,,$$
 the same applies to the ring $ K_\bullet^S(X)$.
 As the 0-order term of  $\sigma_1(N^\ast)$ is 1, $\sigma_1(N^\ast)$ is invertible.
 Generally speaking this does not imply that $j\colon K^\bullet_S(\Xcal) \to K^\bullet_S(X)$ is invertible as 
 $\sigma_1(N^\ast)^{-1}$ may not lie in $K^\bullet_S(\Xcal)$.
\end{rem}
  
  \begin{prop} \label{KS} One defines a product $*$ in $KS(\Xcal)$ by letting 
  $$\cl^S(\Ec_1) * \cl^S(\Ec_2) = \cl^S(\Ec_1\otimes\Ec_2).$$
  \begin{enumerate}
 \item This gives a ring structure to $KS(\Xcal)$ with identity element $\cl^S(\Oc_\Xcal) = \sigma_1(N^\ast)$.
  
\item If $x_1$, $x_2\in KS(\Xcal)$ then 
  $$ x_1 * x_2 = x_1\cdot x_2 \cdot  \sigma_1(N^\ast)^{-1}\quad \text{in}\ K^\bullet_S(X)\otimes\Q.$$
  
 \item  If $f\colon\Xcal\to\Ycal$ is any morphism, and $\Ec$ is a locally free $\Oc_\Ycal$-module, let
  $$ f_S^! (\cl^S\Ec) = \cl^S(f^\ast \Ec).$$
  \end{enumerate}
This  defines a ring homomorphism $f_S^!\colon KS(\Ycal)\to KS(\Xcal)$.
  \end{prop} 
  \begin{proof} The three claims follows from direct computations. 
  
  1.  The associativity follows from the associativity of the tensor product, as
  $$ \left[ \cl^S(\Ec_1) * \cl^S(\Ec_2)\right] * \cl^S(\Ec_3) = \cl^S(\Ec_1\otimes\Ec_2) * \cl^S(\Ec_3) =  \cl^S(\Ec_1\otimes\Ec_2\otimes \Ec_3).$$
  
  Proof that $\sigma_1(N^\ast)$ is the identity element:
$$  \cl^S(\Ec) * \sigma_1(N^\ast)  =   \cl^S(\Ec) * \cl^S(\Oc_\Xcal) = 
    \cl^S(\Ec\otimes \Oc_\Xcal)= \cl^S(\Ec).
$$
  
   2. Let $x_1=\cl^S(\Ec_1)$, $x_2=\cl^S(\Ec_2)$: 
     \begin{align} x_1*x_2 &= \cl^S(\Ec_1) * \cl^S(\Ec_2)  = \cl^S(\Ec_1\otimes\Ec_2) =  \cl((\Ec_1\otimes\Ec_2)\redu)\cdot \sigma_1(N^\ast)  \\
     &= \cl((\Ec_1)\redu\cdot \cl((\Ec_2)\redu) \cdot \sigma_1(N^\ast) = \cl^S(\Ec_1) \cdot \cl^S(\Ec_2) \cdot \sigma_1(N^\ast)^{-1} = x_1 \cdot x_2 \cdot \sigma_1(N^\ast)^{-1}.
       \end{align}
   
   3.
   $$ f_S^!(1) = f_S^!(\cl^S(\Oc_\Ycal)) = \cl^S(f^\ast \Oc_\Ycal) = \cl^S(\Oc_\Xcal) = 1$$
  \begin{align} f_S^!(\cl^S(\Ec_1) *  \cl^S(\Ec_2)) &=  f_S^!(\cl^S(\Ec_1\otimes\Ec_2)) = \cl^S(f^\ast(\Ec_1\otimes\Ec_2)) \\
  &= \cl^S(f^\ast\Ec_1\otimes f^\ast\Ec_2) = \cl^S(f^\ast \Ec_1) * \cl^S(f^\ast \Ec_2)\\ & =  f_S^!(\cl^S(\Ec_1)) *  f_S^!(\cl^S(\Ec_2)).
\end{align}
  \end{proof}

 \subsection{Characteristic classes}
 The definition of the characteristic classes we recall here is the one in terms of the so-called $\gamma$-filtration in K-theory, as in \cite{Ma69} for the ordinary case. However the super version of it displays a few somehow unexpected features.
 
    \begin{defin} The $\Z$-graded ring $GK_S(X)$ is the graded ring of the $\gamma$-filtration $F^\bullet$ of $K^\bullet_S(X)\otimes \Q$.  In other words, $GK_S(X)\simeq GK(X)\otimes_\Q \Q[\Pi]$, where $GK(X)$ is the graded ring of the ordinary $\gamma$-filtration of $K(X)\otimes\Q$. \end{defin}

 We define characteristic classes of  locally free $\Oc_\Xcal$-modules with values in $GK_S(X)$.  If $\Ec$ is a locally free $\Oc_\Xcal$-module, we set $\Ec\redu = \Ec_0\oplus \Ec_1$  (decomposition into the even and the odd parts) and $ \Pi^\Ec = \Pi^{\rk\Pi\Ec_1}$.   
 
 \begin{defin}[Chern classes and Chern character] \label{defchar} One defines:
  \begin{itemize} \item $c_0(\Ec) = \Pi^\Ec$; 
 \item $c_i(\Ec) = \Pi^\Ec c_i(\Ec_0-\Pi\Ec_1) = \Pi^\Ec \gamma^i(\cl \Ec_0 - \cl \Pi\Ec_1-\rk\Ec_0+\rk\Pi\Ec_1)$ {\rm mod} $F^{i+1}$ for $i>0$;
 \item $c_t(\Ec) = \sum_i c_i(\Ec)\,t^i $;
 \item $\ch(\Ec) = \ch(\Ec_0)-\Pi\ch(\Pi\Ec_1)$.
 \end{itemize} 
 \end{defin}
 \begin{rem}   Note that $c_i(\Ec)$ always has the parity of $\rk\Pi\Ec_1$, while in general $\ch_i(\Ec)$ has an even and an odd part. So $\ch_1(\Ec)\ne c_1(\Ec)$ in general.   \end{rem}
 
 \begin{prop}[Properties of characteristic classes] \label{propchar} \ 
 \par
 \begin{enumerate} \item 
 $$ \ch(\Ec) = \sum_{i=1}^{\rk\Ec_0} e^{a_i(\Ec_0)} - \Pi \sum_{i=1}^{\rk\Pi\Ec_1} e^{-\Pi a_i(\Ec_1)}$$
 where $$c_t(\Ec_0)=\prod_{i=1}^{\rk\Ec_0}(1+a_i(\Ec_0)t), \qquad c_t(\Ec_1) =\prod_{i=1}^{\rk\Pi\Ec_1}(\Pi+a_i(\Ec_1)t)\,.
 $$
 \item $\ch (\Oc_\Xcal) =1,\quad \ch(\Pi\Oc_\Xcal) = - \Pi.$
 \item $c_1(\Lcl)=[\cl(\Lcl\redu)]-1$ if $\rk \Lcl=1\vert 0$, while $c_1(\Lcl) = \Pi-[\cl(\Lcl\redu])$ if $\rk\Lcl=0\vert 1$, where $[\ ]$ is the class in $GK_S(X)$ (later on we shall omit writing the  square brackets).
 \item $c_1(\Lcl^\ast)=-c_1(\Lcl)$, $c_1(\Pi\Ec) = -c_1(\Ec) \Pi^{\rk\Ec_0+\rk\Pi\Ec_1}$, $\ch(\Pi\Ec) = -\Pi \ch(\Ec)$.
 \item The Chern classes are functorial with respect to pullbacks,
 $f^\ast c_i(\Ec) = c_i(f^\ast \Ec) $. 
 \item If $ 0 \to \Ec_1\to\Ec\to\Ec_2\to 0$ is exact, then
 $$c_t(\Ec) = c_t(\Ec_1)\,c_t(\Ec_2) \qquad \text{and} \qquad \ch(\Ec) = \ch(\Ec_1) + \ch(\Ec_2)\,.
 $$
 \item $\ch(\Ec_1\otimes\Ec_2) = \ch(\Ec_1)\ch(\Ec_2)$.
 \end{enumerate}
 \end{prop}
 In particular the Chern character defines   a ring morphism
 $  \ch \colon K_S^\bullet ( {\Xcal}) \to GK_S(X)$. We can also define a ``twisted'' Chern character
 $$ \ch^S \colon KS(\Xcal) \to GK_S(\Xcal), \qquad \ch^S (j(x)) = \ch (x).$$

 If we consider the bosonic reduction $X$ as a superscheme, Definition \ref{defchar}  applies  elements of $K^\bullet_S(X)\simeq K_\bullet^S(X)$, and we have $c_i(\Ec)=c_i(\Ec\redu)$ and $\ch(\Ec)=\ch(\Ec\redu)$ for every locally free sheaf $\Ec$ on $\Xcal$. Proposition \ref{propchar} also applies to  $K^\bullet_S(X)$ and we have a ring morphism
$$ \ch \colon K_S^\bullet (X) \to GK_S(X)\,.
$$

 We also  define a  twisted  Chern character $\ch^S \colon K_S^\bullet(X) \otimes\Q \to GK_S(X) $ by letting
 $$  \ch^S(x) = \ch (x\cdot \sigma_1(N^\ast)^{-1}) \,.$$
  Of course we have
 $$ \ch^S(j(x)) = \ch(x)$$
 for $x\in  K_S^\bullet(X)$, and $\ch^S(\cl^S(\Ec))=\ch(\Ec)$ on $KS(\Xcal)$.

 Moreover we consider in $ K_S^\bullet(X)\otimes\Q$ the product 
 $$x_1\ast x_2 = x_1\cdot x_2 \cdot  \sigma_1(N^\ast)^{-1}$$
 Note that this product makes $j$ into a morphism of rings, as 
 $$ j(x)\ast j(y) = xy \cdot \sigma_1(N^\ast)^{2}\cdot \sigma_1(N^\ast)^{-1} =  xy \cdot \sigma_1(N^\ast) = j(xy).$$
 \begin{prop}\ 
 \begin{enumerate}
 \item  The twisted Chern character $\ch^S$ is multiplicative:
 $$ \ch^S(x\ast y) = \ch^S(x)\cdot \ch^S(y)  \quad\text{for all}\ x,y\in  K_S^\bullet(X)\otimes\Q.$$ 
 
\item The twisted  Chern character $\ch^S$  is functorial:  given a morphism $f\colon\Xcal \to \Ycal $ then
 $$ \ch^S(f_S^! (x)) = f^\ast \ch^S(x)\quad \text{for all}\  x\in K_S^\bullet(X)\otimes\Q.$$
 \end{enumerate}
 \end{prop}
 \begin{proof} The proof consists in two direct computations.
 \begin{align}  \ch^S(x\ast y) & =  \ch( x\ast y \cdot   \sigma_1(N^\ast)^{-1}) = \ch(x\cdot y  \cdot  \sigma_1(N^\ast)^{-2}) \\[3pt]
 &= \ch(x\cdot   \sigma_1(N^\ast)^{-1}) \cdot \ch ( y \cdot   \sigma_1(N^\ast)^{-1}) =  \ch^S(x)\cdot \ch^S(y).
 \end{align}
 \begin{align}  \ch^S(f_S^!(y))) & =  \ch(f_S^!(y)\cdot \sigma_1(N^\ast_\Xcal)^{-1}) = \ch\bigl(f^\ast(y \cdot  \sigma_1(N^\ast_\Ycal)^{-1})\cdot  \sigma_1(N^\ast_\Xcal) \cdot  \sigma_1(N^\ast_\Xcal)^{-1}\bigr) \\[3pt]
 &= \ch\bigl(f^\ast(y \cdot  \sigma_1(N^\ast_\Ycal)^{-1})\bigr) = f^\ast \ch    \bigl(y \cdot  \sigma_1(N^\ast_\Ycal)^{-1}\bigr) = f^\ast (\ch^S(y))
  \end{align}
 \end{proof}
 
 \begin{rem} Since $KS(\Xcal) \subset  K^\bullet_S(X)$ we may restrict $\ch^S$ to $KS(\Xcal)\otimes\Q$. 
According to \cite[Prop.\! 20g]{VMP}, $$\ch^S\colon KS(\Xcal)\otimes\Q \to GK_S(X)\otimes \Q$$ is injective whenever the $\gamma$ filtration $F^\bullet_S$ verifies $F^d_S=0$ for $d\gg 0$, for instance when 
 $X$ is regular and projective. \label{reminj}
 \end{rem}
 
 \subsection{Splitting principle} To define the Todd character one uses a version of the splitting principle, which we describe in this section.
 \begin{prop} Let $\Xcal$ be a smooth supervariety, and $\Ec$ a locally free $\Oc_\Xcal$ module. There exists a morphism $f\colon\mathcal Z\to\Xcal$, where $\mathcal Z$ is a smooth supervariety, with the following properties: \begin{enumerate} 
 \item the induced morphism $f^\ast\colon GK_S(X) \to GK_S(Z)$ is injective;
 \item the induced morphism $f_S^! \colon KS(\Xcal) \otimes \Q \to KS(\mathcal Z)\otimes \Q$ is injective;
 \item $f^\ast\Ec$ has a filtration whose quotients are line bundles of rank $1|0$ or $0|1$. 
 \end{enumerate} \label{splitprinc}
 \end{prop}
  \begin{proof} The proof is a straightforward adaptation of the analogous statement for the ordinary case \cite[5.6]{Ma69}. Let $m|n=\rk\Ec$ and $\pi\colon \Ycal=\Ps(\Ec)\to \Xcal$ the projective superscheme associated to $\Ec$ \cite[Def.\,2.15]{BrHRPo20}. If $m\geq 1$, the pullback $\pi^\ast\Ec$ has a quotient $\Oc_{\Ps(\Ec)}(1)$ which is a line bundle of rank $1|0$ and $\pi^\ast\colon K^\bullet_S(X)\to K^\bullet_S(Y)$ is injective. Moreover, from  \cite[8.11]{Ma69}, one has $\pi^\ast(F^i K^\bullet_S(X))= \pi^\ast K^\bullet_S(Y)\cap F^i K^\bullet_S(X)$ for every index $i$; this fact, together with the injectivity of $\pi^\ast\colon K^\bullet_S(X)\to K^\bullet_S(Y)$ implies that $\pi^\ast\colon GK_S(X) \to GK_S(Y)$ is injective as well. 
By Remark \ref{sigmainvert}, also the morphism $\pi_S^! \colon KS(\Xcal) \otimes\Q \to KS(\Ycal)\otimes \Q$ is injective as
$$ \pi_S^!(x) = \pi^\ast ( x\cdot  \sigma_1(N_\Xcal^\ast)^{-1}) \cdot \sigma_1(N_\Ycal^\ast)
\qquad \text{for any} \ x\in KS(\Xcal).$$
Similarly, if $n>1$ and we consider $\pi\colon \Ycal=\Ps(\Pi\Ec)\to \Xcal$, $\pi^\ast\Ec$ has a quotient $\Pi\Oc_{\Ps(\Pi\Ec}(1)$ which is a line bundle of rank $0|1$ and both $\pi^\ast\colon GK_S(X) \to GK_S(Y)$ and  $\pi_S^! \colon KS(\Xcal) \otimes\Q \to KS(\Ycal)\otimes \Q$ are injective. Now  the Proposition is proved by iteration.
 \end{proof}
 
 \subsection{The Todd character}
 We proceed to the definition of the Todd character. 
We define the class $\sigma_1$ of a  locally free $\mathcal O_\Xcal$-module $\Ec$ of 
  rank $r|s$ by letting
 $$ \sigma_1(\Ec) = \cl^\bullet \left[\bigoplus_{i=0}^s \Pi^i\Sym^i\,\Ec\right] \in K_S^\bullet(X)\,.
 $$
 Recall that the standard Todd character is the characteristic class associated with the function
 $$\phi(x)=\frac{x}{1-e^{-x}} = \left[ \sum_{i=1}^\infty  \frac{(-x)^{i-1}}{i!} \right]^{-1}.$$
 This   becomes the definition of the Todd character
 $$\td\colon KS(\Xcal) \to GK_S(X) \otimes\Q$$
  for line bundles of rank $1|0$: if $\ell =\cl^S(\Lcl)$ and $x=c_1(\Lcl)$, then
\begin{equation} \label{toddpari}  \td(\ell) =   \left[ \sum_{i=1}^\infty  \frac{(-x)^{i-1}}{i!} \right]^{-1}.\end{equation}
 If $\rk\Lcl=0|1$ and $\ell =\cl^S(\Lcl)$ we set
 $$ \td(\ell) = \ch \sigma_1(\Lcl^\ast).$$
 
 \begin{lemma} \label{toddline}  The Todd character of line bundles is functorial:
if $f\colon\Xcal \to \Ycal$ is a a morphism, and $\Lcl$ is a line bundle on $\Ycal$, then
$$ \td\circ f_S^! (\cl^S(\Lcl))= f^\ast \circ \td (\cl^S(\Lcl)).$$
  \end{lemma}
  \begin{proof}
   If  $\rk\Lcl=1|0$ then
 $$ \td \circ  f_S^! (\cl^S(\Lcl)) = \td(\cl^S(f^\ast \Lcl ))\,.$$
 This means that we have to compute the Todd character \eqref{toddpari} on 
 $c_1(f^\ast \Lcl)   = f^\ast c_1(\Lcl)$, 
 and the result is 
  $f^\ast \circ \td (\cl^S(\Lcl))$. 
 
 If $\rk\Lcl=0|1$ then
 $$\td \circ  f_S^! (\cl^S(\Lcl)) = \td(\cl^S(f^\ast \Lcl ) )= \ch \sigma_1(f^\ast\Lcl^\ast) = f^\ast\ch\sigma_1(\Lcl^\ast) = f^\ast  \circ \td (\cl^S(\Lcl)).$$
\end{proof}

 Using Lemma \ref{toddline} and the splitting principle, 
 the Todd character is  defined on $KS(\Xcal)$ by assuming that
 \begin{enumerate} \item 
  $ \td(x_1+x_2) = \td x_1 \cdot \td x_2$
  \item $ \td\circ f_S^! = f^\ast \circ \td$ for any morphism $f\colon\Xcal \to \Ycal$.
  \end{enumerate}

\begin{lemma}\label{lem:sigma1} If $\rk\Ec = 0|s$ then
$$ \td \cl^S(\Ec) = \ch \sigma_1(\Ec^\ast).$$
\end{lemma} 
\begin{proof}
 By  the splitting principle (Proposition \ref{splitprinc})  we may  assume that $\Ec$ is a direct sum of line bundles of rank $0|1$, $\Ec = \bigoplus_{i=1}^s \Lcl_i$.
Then 
$$ \Sym^i \,\Ec = \bigoplus_{1\le j_1<\dots<j_i\le s}\Lcl_{j_1}\otimes\dots\otimes \Lcl_{j_i}\quad\text{for}\ i>0, \quad  \Sym^0 \,\Ec = \Oc_\Xcal$$
and 
$$ \sigma_1(\Ec) = 1 +  \bigoplus_{i=1}^s \cl^\bullet\left[\Pi^i 
 \bigoplus_{1\le j_1<\dots<j_i\le s}\Lcl_{j_1}\otimes\dots\otimes \Lcl_{j_i}
\right] =   \prod_{i=1}^s \sigma_1(\Lcl_i)$$
$$ \td \cl^S(\Ec)  =  \prod_{i=1}^s \td \cl^S(\Lcl_i) = \prod_{i=1}^s \ch \sigma_1(\Lcl_i^\ast) \ = \ch \sigma_1(\Ec^\ast).
$$
\end{proof}
  
   \begin{prop}\label{wwwww}
     If $ i\colon X\to\Xcal$ is the canonical embedding,  then
 $$ \ch^S(i^S_! x) = i_\ast (\ch(x)\cdot \td(-\cl^S(N)))$$
 where $N$ is the normal bundle to $X$ in $\Xcal$.
 \end{prop}
 \begin{proof}  Note that $\rk N =(0,n)$ if $\dim\Xcal = (m,n)$ and 
 $$\td(\cl^S(N)) = \ch \sigma_1(N^\ast), \qquad   \td(-\cl^S(N))=\td(\cl^S(N))^{-1} = \ch\sigma_1(N^\ast) ^{-1}.$$
 Moreover $i_\ast \colon GK_S(X) \otimes\Q\to GK_S(X)\otimes\Q$ is the identity.
 Then 
  \begin{align}    i_\ast (\ch(x)\cdot \td(-\cl^S(N)) )&= \ch(x)\cdot  \ch \sigma_1(N^\ast) ^{-1}  \\ 
  &= \ch (i_!^S(x)\cdot \sigma_1(N^\ast) ^{-1}) = \ch^S(i^S_! x).
  \end{align}
 \end{proof}
 
  \subsection{Super Grothendieck-Riemann-Roch}
  
The general formula for the super Grothendieck-Riemann-Roch theorem is the following:
\begin{thm} \label{SGRR}
If $f\colon \Xcal \to \Ycal$ is a morphism of smooth supervarieties, 
such that
$f_{\mbox{\rm\tiny bos}}\colon X \to Y$ is projective, and $x \in K_S^\bullet(X)\otimes \Q$, then
\begin{equation}\label{SGRReq}  \ch^S(f^S_! x) = f_\ast (\ch^S(x) \cdot \td(\cl^S(T_f))\end{equation}
where $T_f$  is the relative tangent sheaf, so that  $\cl^S(T_f )= \cl^S(T_\Xcal) - \cl^S(f^\ast T_\Ycal)$.  
  \end{thm}
  \begin{proof} 
  Note that, since $i^\ast$ is the identity, one has
\begin{align*}
  \td (\cl^S T_\Xcal) &= i^\ast  \td (\cl^S T_\Xcal) =\td(i^!_S(\cl^S T_\Xcal))=\td(T_{\Xcal |\redu})=\td(T_X)\cdot\td(N_\Xcal) \\
 &= \td(T_X)\cdot \ch(\sigma_1(N_\Xcal^\ast))\,,
\end{align*}
where  the last equality follows from Lemma \ref{lem:sigma1}.
  
The same happens for $\Ycal$, so that
  $$
  \td(\cl^S(T_f))= \td(T_{X/Y})\cdot \ch(\sigma_1(N_\Xcal^\ast)\cdot f^\ast(\ch(\sigma_1(N_\Ycal^\ast))^{-1}\,.
  $$
We compute the left-hand side of  equation \eqref{SGRReq}:
 \begin{align} f_\ast (\ch^S(x) \cdot \td(\cl^S(T_f))&= f_\ast \left[\ch(x)\ch \sigma_1(N_X^\ast)^{-1} \td(T_{X/Y})\cdot \ch(\sigma_1(N_\Xcal^\ast)\cdot f^\ast(\ch(\sigma_1(N_\Ycal^\ast)^{-1})\right] \\[3pt]
  &=f_\ast(\ch(x)\td (T_{X/Y})\cdot \ch(\sigma_1(N_\Ycal^\ast)^{-1})
    \end{align}
    Now applying the ordinary Grothendieck-Riemann-Roch theorem to the last term we obtain
$$ f_\ast(\ch(x)\td (T_{X/Y})\cdot \ch(\sigma_1(N_\Ycal^\ast)^{-1})
 = \ch (f_!^S(x)\cdot \sigma_1(N_\Ycal^\ast)^{-1}) = \ch^S (f^S_!(x))\,,
$$
  \end{proof}
    When $f\colon \Xcal \to \operatorname{Spec}\C$  is the structural morphism of $\Xcal$  
  we have $\ch^S\circ f^S_! = \chi^S$ with  $$
 \chi^S(x):= \chi(x_+)-\Pi\cdot \chi(x_-) \in \Z\oplus\Pi\Z\,.
 $$
 where $x=x_+\oplus x_-$ is the decomposition of $x$ into even and odd parts.
 Moreover,  $T_\Ycal=0$ and we obtain the  formula 
  \begin{equation}
  \chi^S(x)  =  f_\ast(\ch^S (x)  \td \cl^S T_\Xcal) = f_\ast (\ch (x) \cdot \td X)
\label{SRRn} \end{equation}
 
 \begin{example} The case $\dim\Xcal=1|n$.  We apply formula \eqref{SRRn} to the case of a split supercurve $\Xcal$  of dimension $1|n$,
 assuming $X$ is  smooth and  projective.
Then $\Jc/\Jc^2 = \Pi\Fc_\Xcal$ for a rank $n$ locally free $\Oc_X$-module $\Fc_\Xcal$, and 
$ \Oc_\Xcal=i_\ast \Lambda^\bullet \Fc_\Xcal$.

Let $\Ec$ be a locally free $\Oc_\Xcal$-module of rank $r\vert s$.   From  \eqref{SRRn} we have 
 \begin{align} \label{chiE} \chi^S(\Ec) & =   f_\ast(\ch^S (\cl_\bullet \Ec)\td T_\Xcal)    \\[3pt] & =
  f_\ast [ (\ch_0(\cl_\bullet \Ec) + \ch_1(\cl_\bullet \Ec)) (1+(1-g)w)]    \\[3pt] & = (1-g) \ch_0(\cl_\bullet \Ec) + \deg \cl_\bullet \Ec 
   \end{align}
  where $w$ is the fundamental class of $X$ and $g$ is genus of $X$, while $\deg \cl_\bullet \Ec$ is the degree of $\ch_1(\Ec)$.
\end{example}

\section{The dimension of the superstack of stable supermaps}
\label{dim}

In this section we use the super Grothendieck-Riemann-Roch formula (Theorem \ref{SGRR}) to formally compute the virtual dimension of the stack of stable supermaps. A rigorous justification of this computation in terms of a perfect obstruction theory and a deformation theory for that moduli superstack will be the object of a future paper \cite{POT}.
 
 Let $\Ycal$ be a smooth superscheme, and fix 
 a class $\beta_0\in A_1(Y)$.   Let $\phi\colon \Xcal \hookrightarrow \Ycal$ be a  closed immersion of smooth superschemes
 such that  the homology class of $\phi_{bos}(X)$ in $Y$ is $\beta_0$. Here 
  $\Xcal$ is a SUSY curve; in particular $\Xcal$ is split and  $\Oc_\Xcal=\Oc_X\oplus\Pi\Lcl$ with  $\deg \Lcl=g-1$. We denote by $\Fc_\Ycal$ the sheaf $\Pi\Jc_\Ycal/\Jc_\Ycal^2$ on the bosonic reduction $Y$ of $\Ycal$, and assume that 
 $\Ycal$ is a  projective smooth superscheme of dimension $r|s$.  
Let 
  $ \beta = (1-\Pi)\beta_0 \in A_1(\Ycal)$; 
 this is the supercycle $\phi_\ast [\Xcal ]$. 
 
 \begin{conj} The virtual  dimension of the moduli superstack  $\Sf\Mf_{g,\nf_{NS},\nf_{RR}}(\Ycal,\beta) $, which parameterizes
 stable supermaps into a smooth superscheme  $\Ycal$, of degree $\beta$, with $ \nf_{NS}$
 Neveu-Schwarz punctures and $nf_{RR}$ Ramond-Ramond punctures, is 
 \begin{align}\label{eq:dimmod}
\operatorname{vdim } \Sf\Mf_{g,\nf_{NS},\nf_{RR}}(\Ycal,\beta)  
  &= (r-3)(1-g) + \nf_{NS}+ \nf_{RR} (1+ s/2) \\ & - \Pi \left[ (1-g)(s-2)+ \nf_{NS}+ (\nf_{RR}/2) (r+1)\right]  \\
  & + (1-\Pi)\int_{\beta_0} [ \ch_1(T_Y)-\ch_1(\Fc_\Ycal)]\,.
\end{align}

 \end{conj}
The core of this conjecture is the assumption  that the deformations of the map $\phi\colon \Xcal \hookrightarrow \Ycal$ when $\Xcal$ is deformed as a SUSY curve
are governed by the exact sequence of sheaves of graded vector spaces 
\begin{equation} \label{seqdef}   0\to \Gc_\Xcal \to \phi^\ast T_\Ycal \to \tilde\Nc_\phi \to 0\,,
\end{equation}
 where $\Gc_\Xcal \hookrightarrow T_\Xcal$ is the sheaf of infinitesimal deformations of $\Xcal$ as a SUSY cuve \cite[Def. 3.10]{BrHR21}.  This implies that the virtual dimension of the superstack $\Sf\Mf_{g}(\Ycal,\beta)$ (we are considering no punctures at the moment) is
 $$ \operatorname{vdim } \Sf\Mf_{g}(\Ycal,\beta) = \chi^S(\tilde\Nc_\phi).$$

 Since $h^0(X,\Gc_\Xcal)=0$ and $h^1(X,\Gc_\Xcal )=3g-3-\Pi (2g-2)$ by \cite[Cor. 3.14]{BrHR21}, we have $$\chi^S(\Gc_\Xcal)=3-3g-\Pi(2-2g)$$ 
and from equation \eqref{chiE}
\begin{equation} \label{chiTY}  \chi^S(\phi^\ast T_\Ycal) = (1-g)(r-\Pi s) +  (1-\Pi) \int_{\beta_0} [ \ch_1(T_Y)-\ch_1(\Fc_\Ycal)] \end{equation}
so that
\begin{align}    \operatorname{vdim } \Sf\Mf_{g}(\Ycal,\beta)  & =  \chi^S(\tilde N_\phi)   \\  & =  (1-g)(r-3)-\Pi(1-g)(s-2) + (1-\Pi) \int_{\beta_0} [ \ch_1(T_Y)-\ch_1(\Fc_\Ycal)]\,. \label{expdimmod} \end{align}
 
 When the target $\Ycal$ is bosonic (an ordinary scheme $Y$) then $s = \ch_1(\Fc_\Ycal) =0$ and the previous formula reduces
 to
$$   \operatorname{vdim } \Sf\Mf_{g}(Y,\beta_0 ) =     (1-g)(r-3) - \Pi(2g-2) + (1-\Pi) \int_{\beta_0}  \ch_1(T_Y)  \label{expdimmod2} $$
which agrees with the formula given in \cite[Thm.~4.3.8]{KSY-2022}.

  We consider now the case where $\Xcal$ is a SUSY curve with $\nf_{NS}$ Neveu-Schwarz  punctures and $\nf_{RR}$ Ramond-Ramond punctures. This is the relevant case in the study of the moduli stack of stable (SUSY) supermaps whose structure curve has  Neveu-Schwarz  and  Ramond-Ramond punctures.  
Now we have $\deg \Fc_\Xcal= g-1+\nf_{RR}/2$ and then Equation \eqref{chiTY} is replaced by
\begin{equation}\label{eq:chiS2}
\chi^S(\phi^\ast T_\Ycal)= (1-g)(r-\Pi s) +(\nf_{RR}/2) (s-\Pi r) + (1-\Pi) \int_{\beta_0} [ \ch_1(T_Y)-\ch_1(\Fc_\Ycal)] \,.
\end{equation}

We     consider the sequence \eqref{seqdef}
 but 
  where now $\Gc_\Xcal \hookrightarrow T_\Xcal$ is the sheaf of infinitesimal deformations of $\Xcal$ as a SUSY cuve with Neveu-Schwarz and Ramond-Ramond punctures  \cite{LeRoth88}\cite[Def. 3.10]{BrHR21}. Since $h^0(X,\Gc_\Xcal)=0$ and $$h^1(X,\Gc_\Xcal )=3g-3+\nf_{NS}+\nf_{RR}-\Pi (2g-2+\nf_{NS}+\nf_{RR}/2)$$ by \cite[Cor. 3.14]{BrHR21}, we have 
  $$
\chi^S(\Gc_\Xcal)=3-3g-\nf_{NS}-\nf_{RR}-\Pi(2-2g-\nf_{NS}-\nf_{RR}/2)
  $$
and from here one obtains equation \eqref{eq:dimmod}. 
In particular the even virtual dimension is
\begin{align}\operatorname{even} \operatorname{vdim } \Sf\Mf_{g,\nf_{NS},\nf_{RR}}(\Ycal,\beta) 
& = (r-3)(1-g) + \nf_{NS}+ \nf_{RR} (1+ s/2)
\\ & + \int_{\beta_0} [ \ch_1(T_Y)-\ch_1(\Fc_\Ycal)]\,.
\end{align}
 
\begin{example} \label{notprop} Assume that $\Ycal=\Ps^{r|s}$ is a projective superspace  (see \cite[Ch.~4, \S3.4]{Ma88}, \cite[Def.~2.6]{BrHRPo20}). The sheaf $\Fc_\Ycal$ is the direct sum  of $s$ copies of  $\Oc_{\Ps^r} (-1)$. 
Let us denote $\beta_0=dH$, where $H$ is the  cycle of a line.
Then $\beta=(1-\Pi)dH$,  and
$\int_{\beta_0} \ch_1(\Fc_\Ycal)= -sd$, so that  
\begin{align}\label{eq:dimmod2}
\operatorname{vdim } \Sf\Mf_{g,\nf_{NS},\nf_{RR}}(\Ps^{r|s},\beta)   &=   (r-3)(1-g) + \nf_{NS}+\nf_{RR} + \int_{\beta_0} \ch_1(T_{\Ps^r}) + s (d + \nf_{RR}/2)\\[-7pt]
&-\Pi \left[ (1-g)(s+2)+ \nf_{NS}+ (\nf_{RR}/2) (r+1) +s d +\int_{\beta_0} \ch_1(T_{\Ps^r}) \right]\,.
\end{align}

Thus, the expected dimension $D$ of the bosonic reduction of the  superstack  $\Sf\Mf_{g,\nf_{RR}}(\Ps^{r|s},\beta) $ of stable supercurves with target $\Ps^{r|s}$ and class $\beta=(1-\Pi)dH$  is
\begin{align}
D & = (r-3)(1-g) +\nf_{NS}+ \nf_{RR} + \int_{\beta_0}\ch_1(T_{\Ps^r}) + s (d + \nf_{RR}/2)
\\
& = \dim \Mf^{spin}_{g,\nf_{NS},\nf_{RR}}(\Ps^r,d) + s (d + \nf_{RR}/2) = \dim \Mf^{spin}_{g,\nf_{NS},\nf_{RR}}(\Ps^r,d) +s \chi(X, \F_\Xcal(1))\,,
\end{align}
where $\Mf^{spin}_{g,\nf_{NS},\nf_{RR}}(\Ps^r,d)$ is the stack of stable spin maps   to $\Ps^r$ with 
genus $g$, $\nf_{NS}$ Neveu-Schwarz punctures, 
$\nf_{RR}$  Ramond-Ramond  punctures,  and   degree $d$.

 This formula  confirms that, in accordance with  Subsection \ref{bosonic}, the bosonic reduction of the superstack  $\Sf\Mf_{g,\nf_{RR}}(\Ps^{r|s},\beta)$   in general  \emph{is not}   $ \Mf^{spin}_{g,\nf_{RR}}(\Ps^r,d)$.  
In particular, the superstack $\Sf\Mf_{g,\nf_{RR}}(\Ps^{r|s},\beta)$ of stable supermaps with target $\Ps^{r|s}$ and class $\beta=(1-\Pi) dH$ \emph{is not   proper}, unless $s=0$ (the target is $\Ps^r$) or $d= \nf_{RR} = 0$.
\end{example}

\small

\def\cprime{$'$}


\end{document}